\documentclass[10pt,a4paper]{article}
\usepackage{amsmath}
\usepackage{amsthm}
\usepackage{amsfonts}
\usepackage{amssymb}
\usepackage{indentfirst}
\usepackage{marginnote}
\usepackage{xcolor}
\usepackage[left=3cm,top=3cm,right=3cm,bottom=3cm]{geometry}

\newtheorem{thm}{Theorem}[section]
\newtheorem{lem}[thm]{Lemma}
\newtheorem{co}[thm]{Corollary}
\newtheorem{pr}[thm]{Proposition}
\newtheorem{re}{Remark}[section]
\numberwithin{equation}{section}

\begin{document}
\title{Large and moderate deviations for a $\mathbb{R}^d$-valued branching random walk with a random environment in time }
\author{Chunmao Huang$^{a}$, Xin Wang$^{a}$, Xiaoqiang Wang$^{b}$\footnote{Corresponding author at: Shandong University (Weihai), School of Mathematics and Statistics, 264209, Weihai, China.  \newline \indent \ \ Email addresses: xiaoqiang.wang@sdu.edu.cn (Xiaoqiang Wang).}
\\
\small{\emph{$^{a}$Harbin Institute of Technology (Weihai), Department of Mathematics, 264209, Weihai, China}}\\
\small{\emph{$^{b}$Shandong University (Weihai), School of Mathematics and Statistics, 264209, Weihai, China}}}
\date{}
\maketitle

\begin{abstract}
 We consider a $\mathbb{R}^d$-valued discrete time branching random walk with a stationary and ergodic environment in time.  Let $Z_n$ be the counting measure of particles of generation $n$. With the help of the uniform convergence of martingales and   multifractal analysis, we show a large deviation result associated to the measures $Z_n$ as well as the corresponding moderate deviations.


\end{abstract}

%%%%%%%%%%%%%%%%%%%%%%%%%%%%%%%%%  Section 1 %%%%%%%%%%%%%%%%%%%%%%%%%%%%%%%%%%%
\section{Introduction}\label{LTS1}

\noindent
In this paper, we consider a stochastic process called  \emph{branching random walk with a random environment in time} (BRWRE) which is a generalization of the classical branching random walks (see e.g. \cite{biggins77, biggins, %biggins97, %chauvin,
ka,na2014}) by considering the influence of the time. In this model, particles produce offspring which scatter in the real space $\mathbb R^d$ and the reproduction is affected by a time-dependent random environment. The distributions of the point processes  formulated by the number of particle's offspring and their displacements vary from generation to generation  according to a time random environment%, but the influence of the location (i.e. the space environment) is not considered
.   As an application, such model can be used to describe   the distribution of plants in the space  %or  to analyse the traits of biological individuals
along genealogies.
First proposed by Biggins and Kyprianou  \cite{biggins04}, this model and its transformation such as weighted branching processes in random environments (see e.g. \cite{ku}) and multiplicative cascades  in random environments (see e.g. \cite{liang}) attract  many authors' attention.
Compared with the other models of  branching random walks in random environments studied largely in the literature, see e.g. \cite{greven,BCG93,CP071, CP,CY,HN09, MN, yosh,OWO13}, this model does not consider the influence of space environments,  and moreover, it considers  particles walking in the real space $\mathbb R^d$ rather than  on the integer lattice $\mathbb Z^d$.
Recently, the  BRWRE obtained much  progress on research, see e.g. \cite{GLW14, GL16,huang2,  MP, huang17}.
However, these studies only involved the case where the space dimension $d=1$, namely, the walks were restricted to %proceed on
a straight line, hence we can not use the results in the literature  to deal with the higher dimensional case, such as the application example that plants are distributed on a two-dimensional plane.
Therefore, whether for theoretical or practical purposes, the study on the $\mathbb{R}^d$-valued BRWRE with $d>1$ is demanded. For branching random walks, one of the main research interests is to discover the  asymptotic properties of the counting measure of particles in the $n$-th generation. On this subject, Huang \emph{et al.} \cite{huang2} and Wang and Huang \cite{huang17} established a large deviation principle and a moderate deviation principle respectively associated to the $n$-th generation counting measure  for {the} BRWRE in $\mathbb R$ (with dimension $d=1$). When the space dimension increases, many methods and techniques for one-dimensional case are no longer applicable. Facing such challenge, this paper {aims}  to investigate the large and moderate deviations of  the $n$-th generation counting measure for  BRWRE in general  real space  $\mathbb R^d$ with  dimension $d\geq1$.

%As a generalization of the classical branching random walks (see e.g. \cite{biggins77, biggins, biggins97, chauvin,ka,na2014}),  \emph{Branching random walk with a random environment in time} (BRWRE) characterises   branching random walks influenced by time: the distributions of the point processes, which were indexed by  particle $u$ and formulated by the number of its offspring and their displacements, vary from generations according to a time random environment. First proposed by Biggins and Kyprianou  \cite{biggins04}, BRWRE has been under wider and wider investigation recently, see e.g. \cite{GLW14, GL16,huang2,  MP, huang17}. Compared with the other models of  branching random walks in random environments studied largely in the literature, see e.g. \cite{greven,BCG93,CP071, CP,CY,HN09, MN, yosh,OWO13}, this model does not consider the influence of space environments,  moreover, it considers  particles walking in the real space $\mathbb R^d$ rather than  the integer lattice $\mathbb Z^d$.  Recently, for {the} BRWRE in $\mathbb R$ (with dimension $d=1$),  Huang \emph{et al.} \cite{huang2} established a large deviation principle for the counting measure of particles of generation $n$, while Wang and Huang \cite{huang17} showed a corresponding moderate deviation principe. This paper {aims}  to generalize such results  to {the}  high dimensional real space  $\mathbb R^d$ with  dimension $d\geq1$.

%\bigskip

Let us describe the model of BRWRE in detail. The random environment in time, denoted by $\xi=(\xi_n)$,  is a stationary and ergodic sequence  of random variables, indexed by the time $n\in\mathbb{N}=\{0,1,2,\cdots\}$.%, taking values in some measurable space $(\Theta, \cal E)$.  Without loss of generality\XW{}{,} we can suppose that $\xi$ is defined on the product space
%$(\Theta^{\N}, \cal E^{\otimes\N}, \tau) $, with  $\tau$  the law of $\xi$.
Each realization of $\xi_n$ corresponds to a distribution $\eta_n=\eta(\xi_n)$
 on $\mathbb{N}\times \mathbb{R}^d\times \mathbb{R}^d\times\cdots$, where $d\geq 1$ is  the dimension of the real space.
Given the environment $\xi=(\xi_n)$, the process can be described as follows:

$\bullet$ At time 0, one initial particle ${\emptyset} $ of generation 0 is located at $S_{\emptyset}=(0,...,0)\in \mathbb{R}^d$;

$\bullet$ At time 1, $\emptyset$ is replaced by  $N=N({\emptyset})$ particles of generation 1, located at $L_i=L_i(\emptyset)=(L^1_i,L^2_i,...,L^d_i)$, $1\leq i\leq N$, where the random vector $X(\emptyset)=(N,L_1,L_2,...)\in \mathbb{N}\times \mathbb{R}^d\times \mathbb{R}^d\times\cdots$ is of distribution $\eta_0=\eta(\xi_0)$.

$\bullet$ In general, each particle $u$ of generation $n$ located at $S_u=(S_u^1,\cdots,S_u^d)\in\mathbb R^d$ is replaced at time {n+1} by $N(u)$ new particles $ui$ of generation $n+1$, located at
$$S_{ui}=S_u+L_i(u)\qquad(1\leq i\leq N(u)),$$
where the  random vector $X(u)=(N(u),L_1(u),L_2(u),\cdots)$ is of distribution $\eta_n=\eta(\xi_n)$.  Note that the values $L_i(u)$ for $i> N(u)$ do not play any role for our model; we introduce them only for convenience. All particles behave independently conditioned on the environment $\xi$.%{What does this sentence mean?}

For each realization $\xi $ of the environment sequence,
let $(\Gamma, {\cal G},  \mathbb{P}_\xi)$ be the probability space under which the
process is defined. %(when the environment $\xi$ is fixed to the given realization).
The probability
$\mathbb{P}_\xi$ is usually called  quenched law.
%Let $\tau$ be the law of the environment $\xi\in\Theta^\mathbb N$, where $\mathbb N=\{1,2,\cdots\}$.
%The total probability space can be formulated as the product space
%$( \Theta^{\mathbb{N}}\times\Gamma , {\cal E}^{\N} \otimes \cal G,   \mathbb{P})$,
% where $ \mathbb{P} = \E  (\delta_\xi \otimes \mathbb{P}_{\xi}) $ with $\delta_\xi $ the Dirac measure at $\xi$ and $\E$ the expectation with respect to the law of  $\xi$, so that  for all measurable and
% positive function $g$ defined on $\Theta^{\mathbb{N}}\times\Gamma$, we have
 % $$\int_{ \Theta^{\mathbb{N}}\times\Gamma } g (x,y) d\mathbb{P}(x,y) = \E  \int_\Gamma g(\xi,y) d\mathbb{P}_{\xi}(y).$$
The total
probability $\mathbb{P}$ is usually called annealed law.
The quenched law $\mathbb{P}_\xi$ may be considered to be the conditional
probability of $\mathbb{P}$ given $\xi$. The expectation with respect to $\mathbb{P}$ (resp. $\mathbb P_\xi$) will be denoted by $\mathbb E$ (resp. $\mathbb E_\xi$).

\medskip
Let $\mathbb{U}=\{\emptyset\}\cup_{n\geq1}\mathbb{N^*}^n$ be the set of all finite sequences $u=u_1\cdots u_n$ and $\mathbb I={\mathbb N^*}^{\mathbb N^*}$ be the set of all infinite sequences, where $\mathbb N^*=\{1,2,\cdots\}$.  For $u\in\mathbb{U}$ or $\mathbb I$, we write $u$ for the length of $u$, and $u|n$ for the restriction to the first $n$ terms of $u$, with the convention that   $u|0=\emptyset$.
By definition, under $\mathbb{P}_\xi$, the random vectors $\{X(u)\}$, indexed  by $u\in\mathbb{U}$, are independent of each other, and each  $X(u)$ has distribution $\eta_n=\eta(\xi_n)$ if $|u|=n$.

Let $\mathbb{T}$ be the Galton-Watson tree  with defining element $\{N(u)\}$ %. We have: (a) $\emptyset\in\mathbb{T}$; (b) if $u\in\mathbb{T}$, then $ui\in\mathbb{T}$ if and only if $1\leq i\leq N(u)$; (c) $ui\in\mathbb{T}$ implies $u\in\mathbb{T}$.
and   $\mathbb{T}_n=\{u\in\mathbb{T}: |u|=n\}$ be the set of particles of generation $n$. %Here the null sequence $\emptyset$ (of length $0$), which represents the initial particle, can be regarded as the root of the tree $\mathbb T$; $ui$  represents the $i$th child of the particle $u$; $N(u)$  represents  the number of   offspring of  $u$.
%For the Galton-Watson tree $\mathbb T$, the boundary of $\mathbb T$ is defined as
%$$\partial \mathbb T=\{u\in \mathbb I: u|n\in \mathbb T \text{ for all }n\in \mathbb N\}.$$
%As a subset of $\mathbb I={\mathbb N^*}^{\mathbb N^*}$, $\partial \mathbb T$ is a metrical and compact topological space with
%$$[u]=\{w\in\partial \mathbb T: w||u|=u\},\qquad u\in\mathbb T$$
%its topological basis, and with the standard ultrametric distance
%$$d(u,v)=e^{-\sup{\{|w|:\; u,v\in [w]\}}}.$$
For $n\in\mathbb N$, let
\begin{equation}
Z_n(\cdot)=\sum_{u\in{\mathbb{T}_n}}\delta_{S_u}(\cdot)
\end{equation}
be the counting measure of particles of generation $n$. %For a measurable  subset $A$ of $\mathbb R^d$,
%$$Z_n(A)=\#\{u\in\mathbb T_n: S_u\in A \}$$
%presents the number of particles of generation $n$ located in $A$.
In this paper, we are interested in large and moderate deviations associated to the sequence of measures $\{Z_n\}$.

%For $n\in\mathbb{N}$ and $z=x+\mathbf iy\in\mathbb{C}^d$, where $x,y\in\mathbb{R}^d$, we denote the Laplace transform of $Z_n$ by
%\begin{equation}\tilde{Z}_n(z)=\int e^{\langle z,\;\omega\rangle}Z_n(d\omega)=\underset{u\in\mathbb{T}_n}\sum e^{\langle z,S_u\rangle},\end{equation}

For $n\in\mathbb{N}$ and $z=x+\mathbf iy\in\mathbb{C}^d$ (later and throughout the paper we use $x$ and $y$ to represent the real and imaginary parts of $z\in\mathbb{C}^d$ respectively, i.e. $x=\text{Re} z, y=\text{Im} z\in\mathbb R^d$,  while we use $t$ to represent real vector in $\mathbb R^d$), put
\begin{equation}
m_n(z)=\mathbb{E}_\xi\sum^{N(u)}_{i=1}e^{\langle z,{L_i(u)}\rangle}\quad(|u|=n),
\end{equation}
where  $\langle\cdot,\cdot\rangle$ is the notation of inner product that $\langle z_1,z_2\rangle=\sum\limits_{i=1}^dz_1^i\bar z_2^i$ if $z_1=(z_1^1,\cdots,z_1^d)$ and $z_2=(z_2^1,\cdots,z_2^d)\in\mathbb C^d$.
 Throughout the paper, we assume that
\begin{equation}\label{ASS}
 \qquad  N\geq1 \qquad\text{and} \qquad\mathbb P(N=1)<1, % \mathbb E\left[\frac{N}{m_0(0)}\log^+N\right]<\infty.
\end{equation}
which means that one particle produces at least one child in general and at least two children with positive probability.
We can see that the assumption (\ref{ASS})  leads to $\mathbb E\log m_0(0)>0$,
 which means that the corresponding  branching process in a random environment (BPRE), $\{Z_n(\mathbb R^d)\}$, is \emph{supercritical};  and  furthermore, the assumption (\ref{ASS}) in fact  ensures that $Z_{n}(\mathbb{R}^d)$ goes to infinity  almost surely (a.s.).
%the limit of the normalized population, $W(0)$, is non-degenerate. Precisely,
%\begin{equation*}
%\mathbb{P}_{\xi }(W(0)>0)=\mathbb{P}_{\xi }(Z_{n}(\mathbb{R}^d)\rightarrow \infty )=\lim_{n\rightarrow
%\infty }\mathbb{P}_{\xi }(Z_{n}(\mathbb{R}^d)>0)=1\quad a.s..
%\end{equation*}
We refer to \cite{a,a1, tanny1, tanny2} for more information about BPRE.

\medskip
We first show a large deviation result  for the sequence of measures $\{Z_n(n\cdot)\}$.
Define the function
\begin{equation}
\Lambda(t)=\mathbb E\log m_0(t) \qquad(t\in\mathbb{R}^d)
\end{equation}
and set $\Omega_{\Lambda}=\{t\in\mathbb R^d: \text{$\Lambda(t)$ is well defined as real numbers and differentiable}\}$. For simplicity,
we  assume   $\Omega_{\Lambda}=\mathbb R^d$, so that
$\Lambda(t)$ is well defined  and differentiable everywhere on $\mathbb{R}^d$ (otherwise later one needs to consider the interior of $\Omega\bigcap\Omega_\Lambda$ instead of  $\Omega$  in Theorem \ref{Conver1.12}). Denote the Legendre transform of the function $\Lambda$ by
\begin{equation*}
\Lambda^*(\alpha)=\underset{t\in\mathbb{R}^d}\sup\left\{\langle t,\alpha\rangle-\Lambda(t)\right\}\qquad (\alpha\in\mathbb R^d).
\end{equation*}
It can be seen that $\Lambda^*(\alpha)=\langle t,\nabla\Lambda(t)\rangle-\Lambda(t)$ if $\alpha=\nabla\Lambda(t)$.
For $\varepsilon>0$, denote $$D(z,\varepsilon)=\{\zeta\in\mathbb{C}^d:|\zeta^j-z^j|<\varepsilon,\forall j=1,\cdots,d\},$$$$ B(t,\varepsilon )=\{s\in\mathbb{R}^d:\|s-t\|\leq\varepsilon\}$$
the neighbourhood of $z\in\mathbb C^d$ and that of $t\in\mathbb R^d$ respectively, where $\|a\|=\sum\limits_{i=1}^d|a^i|^2$ for $a=(a^1,\cdots,a^d)\in\mathbb C^d$%can be understood as the distance from $a$ to the origin
. Furthermore, set
$$\alpha_0(z,\varepsilon)=\underset{\zeta\in D(z,\varepsilon)}\inf|m_0(\zeta)|,$$
$$I=\{t\in\mathbb{R}^d:\langle t,\nabla\Lambda(t)\rangle-\Lambda(t)<0\},$$
$$\Omega_1=\text{int}\{t\in\mathbb{R}^d:\mathbb{E}\log^+\mathbb{E}_\xi\tilde{Z}_1(t)^p<\infty~for~some~p>1\},$$
\begin{equation*}
\Omega_2=\{t\in\mathbb{R}^d:\exists\delta_t>0\;\;\text{such that}\;\;
\mathbb{E}\log^-\alpha_0(t,\delta_t)<\infty\},
\end{equation*}
$$\text{ $\Omega=I\bigcap\Omega_1\bigcap\Omega_2$ }.$$

%We first show a large deviation result for the sequence of measures $\{Z_n(n\cdot)\}$.
\begin{thm}[Large deviations]\label{Conver1.12}
 Let $ \mathcal{J}=\{\nabla\Lambda(t)\in\mathbb{R}^d:t\in\Omega\}$ and $ \tilde{\mathcal J}=\{\alpha\in\mathbb{R}^d:\Lambda^*(\alpha)<0\}$. Then
with probability 1, for all measurable  $A\subset\mathcal J$,
\begin{eqnarray}\label{LDMDt1e1}
-\inf_{\alpha\in \mbox{\emph{int}} A }\Lambda^*(\alpha)
\leq\liminf_{n\rightarrow\infty}\frac{1}{n}\log Z_n(nA)
\leq\limsup_{n\rightarrow\infty}\frac{1}{n}\log Z_n(nA)
\leq
-\inf_{\alpha\in  \bar A}\Lambda^*(\alpha),
\end{eqnarray}
where   $\bar A$ denotes the closure of $A$. If additionally  $\mathbb E[\mathbb P_\xi (\|L_1\|\leq a)^{-1}]<\infty$ for some constant $a>0$, then (\ref{LDMDt1e1}) holds for all  measurable  $A\subset\tilde{\mathcal J}$.
\end{thm}

Theorem \ref{Conver1.12} describes a large deviation property about the the sequence of measures $\{Z_n(n\cdot)\}$, which generalizes the result of Huang  \emph{et al} \cite{huang2} for  one-dimensional case ($d=1$). In \cite{huang2}, the authors established a large deviation principle by calculating the limit of the free energy, with the help of the positions of the extreme points of generation $n$. The number of such extreme points are just two (the leftmost one and the rightmost one) for the case $d=1$, but it is infinite for  the case $d>1$. Therefore, their method is no longer suitable for high dimensional case.
Inspired by Attia  \cite{na}%(or Attia and Barral \cite{na2014})
, we  prove  Theorem \ref{Conver1.12} by calculating directly the upper and lower bounds, via the approach of multifractal analysis. Although this basic method is traditional, the flash point of our proof  is reflected in the treatment of the environment. Due to the existence of the environment, especially the stationary and ergodic (rather than  i.i.d) environment, many details during the proof are intractable.

\begin{re}\label{Conver1.12c1}It can be seen that $\mathcal J\subset\tilde{\mathcal J}$.  If $\mathbb E[\mathbb P_\xi (\|L_1\|\leq a)^{-1}]<\infty$ for some constant $a>0$, it follows  from   Theorem \ref{Conver1.12} that with probability 1, for all $\alpha\in\mbox{\emph{int}}\tilde{\mathcal{J}}$,
\begin{equation}\label{LDMDS1Cr1}
\lim_{\varepsilon\rightarrow0^+}\lim_{n\rightarrow\infty}\frac{1}{n}\log Z_n(nB(\alpha,\varepsilon))=\Lambda^*(\alpha).
\end{equation}
\end{re}

For the deterministic environment case, the formula (\ref{LDMDS1Cr1}) was shown in \cite{na,na2014}, and it also can be deduced from \cite{b92}.

%%%%%%%%%%%%%%%%%%%%%%%%%%%%%%%%%%%%%%%%%%%%%%%%%%%%%%%%%%%%%%%%%%%%%%%%%%%%%%%

\medskip

Now we show moderate deviations associated to $\{Z_n\}$. Write $\pi_{0}=m_{0}(0)$. For one-dimensional case, under the  assumption
\begin{equation}\label{LDMDAS}
\mathbb{E}_\xi\underset{u\in \mathbb  T_1}{\Sigma}S_u=0\;\; a.s.\quad\text{and}\quad\text{$\mbox{ess}\sup\frac{1}{\pi_0}\mathbb{E}_\xi\underset{u\in \mathbb T_1}{\Sigma}e^{\delta\|S_u\|}<\infty$  for some $\delta>0$,}
\end{equation}
Wang and Huang \cite{huang17} established a moderate deviation principle for the sequence of measures $\frac{Z_n(a_n\cdot)}{Z_n(\mathbb{R}^d)}$,
where $(a_n)$ is a sequence of positive numbers satisfying $\frac{a_n}{n}\rightarrow0$ and $\frac{a_n}{\sqrt{n}}\rightarrow\infty$. Here we want to weaken the strong assumption (\ref{LDMDAS}), but the cost is that the environment $\xi$ is required to  satisfy the classical mixing conditions of Ibragimov ($\phi$-mixing) or of Rosenblatt (strong mixing) and moreover, the claim to $a_n$ will also be strengthened.  Let $\phi: \mathbb N\rightarrow[0,1]$ be nonincreasing and
 the sequence $\xi=(\xi_n)$ is called to be $\phi$-mixing (see \cite{ibr62}) if for each $k$ and $n$,
\begin{equation}\label{LDMDAm1}
 |\mathbb P(E_1\cap E_2)-\mathbb P(E_1)\mathbb P(E_2)|\leq\phi(n)\mathbb P(E_1)
 \end{equation}
 for all $E_1\in\sigma(\xi_1,\cdots,\xi_k)$ and $E_2\in\sigma(\xi_{k+n},\xi_{k+n+1},\cdots )$. Let $\rho:\mathbb N\rightarrow[0,\frac{1}{4}]$ and the sequence $\xi=(\xi_n)$ is called to be strong mixing (see \cite{robs56}) with mixing coefficients $\rho(n)$ if for each $k$ and $n$,
\begin{equation}\label{LDMDAm2}
  |\mathbb P(E_1\cap E_2)-\mathbb P(E_1)\mathbb P(E_2)|\leq\rho(n)
   \end{equation}
for all $E_1\in\sigma(\xi_1,\cdots,\xi_k)$ and $E_2\in\sigma(\xi_{k+n},\xi_{k+n+1}, \cdots )$.

\begin{thm}[Moderate deviations]\label{Mdp1.4}
Let $(a_n)$ be a sequence of positive numbers satisfying
\begin{equation}\label{LDMDan}
0<\liminf_{n\rightarrow\infty}\frac{a_n}{n^\alpha}\qquad\text{and}\qquad\limsup_{n\rightarrow\infty}\frac{a_n}{n^\beta}<\infty
\end{equation}
for some $\alpha,\beta\in(\frac{1}{2},1)$.
Assume that the environment $\xi=(\xi_n)$ satisfies (\ref{LDMDAm1}) with $\sum_n\phi(n)^{1/\theta}<\infty$ for some $\theta>1$ or  (\ref{LDMDAm2}) with
$\sum_n\rho(n)^{1/\theta}<\infty$ for some $\theta>2$, and the following assumption holds:
\begin{equation}\label{LDMDAS2}
\mathbb{E} \frac{1}{\pi_0}\underset{u\in \mathbb  T_1}{\Sigma}S_u=0 \quad\text{and}\quad\text{$ \mathbb E\frac{1}{\pi_0} \underset{u\in \mathbb T_1}{\Sigma}e^{\delta\|S_u\|}<\infty$  for some $\delta>0$.}
\end{equation}
If $0\in \Omega_1\bigcap\Omega_2$, then with probability $1$,
for all measurable $A\subset\mathbb{R}^d$,
\begin{eqnarray*}
-\underset{x\in \mbox{\emph{int}}A}{\inf}\Gamma^*(x)
\leq\underset{n\rightarrow\infty}{\liminf}\frac{n}{a^2_n}\log\frac{Z_n(a_nA)}{Z_n(\mathbb{R}^d)}
\leq\underset{n\rightarrow\infty}{\limsup}\frac{n}{a^2_n}\log\frac{Z_n(a_nA)}{Z_n(\mathbb{R}^d)}
\leq-\underset{x\in \bar{A}}{\inf}\Gamma^*(x),
\end{eqnarray*}
%a.s. on the survival event $\{Z_n(\mathbb R^d)\rightarrow\infty\}$,
where  $\bar A$ denotes the closure of $A$,  and the rate function $\Gamma^*(x)=\underset{t\in\mathbb{R}^d}\sup\left\{\langle t,x\rangle-\Gamma(t)\right\}$ is the Legendre transform of
$\Gamma(t)=\frac{1}{2}\langle t, \mathbf{C}t\rangle$ ($t\in\mathbb R^d$),
where $\mathbf{C}=(c_{ij})$ is the matrix with elements $c_{ij}=\mathbb{E}\left[\frac{1}{\pi_0}\underset{u\in \mathbb T_1}{\Sigma}(S^i_u-\frac{1}{\pi_0}\mathbb E_\xi \underset{u\in \mathbb T_1}{\Sigma}S^i_u)(S^j_u-\frac{1}{\pi_0}\mathbb E_\xi \underset{u\in \mathbb T_1}{\Sigma}S^j_u)\right]
%=\mathbb E\left[\frac{1}{\pi_0}\underset{u\in \mathbb T_1}{\Sigma}S_u^iS_u^j-\frac{1}{\pi_0^2}(\mathbb E_\xi\underset{u\in \mathbb T_1}{\Sigma}S_u^i)(\mathbb E_\xi\underset{u\in \mathbb T_1}{\Sigma}S_u^j)\right]
$.
\end{thm}

\begin{re} If the matrix $\mathbf{C}$ is invertible,    denoting its inverse by $\mathbf{C}^{-1}$, then we have $\Gamma^*(x)=\frac{1}{2}\langle x,\mathbf{C}^{-1}x\rangle$.
\end{re}

Theorem \ref{Mdp1.4} generalizes and improves the result of Wang and Huang \cite{huang17} for the case $d=1$. Clearly,  the assumption (\ref{LDMDAS2}) is   weaker than (\ref{LDMDAS}), but the restrictions on $(a_n)$ in Theorem \ref{Mdp1.4} make it impossible to
deal with the case where  $\frac{a_n}{n}\rightarrow0$ and $\frac{a_n}{\sqrt{n}}\rightarrow\infty$ but  (\ref{LDMDan}) is not valid, for example $a_n=\frac{n}{\log n}$ or  $a_n=\sqrt{n}\log n$.

\medskip
The rest part of the paper is organised as follows. Firstly, in Section  \ref{LDMDS2}, we introduce the  natural martingale in BRWRE and study its uniform convergence so as to make preparations for later proofs. Then from Sections \ref{LDMDS3} to \ref{LDMDSS3}, we work on  multifractal analysis and  large deviations: in Section  \ref{LDMDS3}, we state the results of  multifractal analysis (see Theorem \ref{Conver1.12ma} and Corollary \ref{Conver1.12cma}) and use them to prove Theorem  \ref{Conver1.12}; Section \ref{LDMDS31} and  Section \ref{LDMDSS3} are devoted to the proofs of Theorem \ref{Conver1.12ma} and Corollary \ref{Conver1.12cma} respectively. Finally in Section \ref{LDMDS4}, we consider     moderate deviations and  give the proof of Theorem \ref{Mdp1.4}.

%%%%%%%%%%%%%%%%%%%%%%%%%   section 2   %%%%%%%%%%%%%%%%%%%%%%%%%%%%%%%%%%%%%%
\section{Uniform convergence of martingale}\label{LDMDS2}

We start with the introduction of the natural martingale in BRWRE.
For $n\in\mathbb{N}$ and $z=x+\mathbf iy\in\mathbb{C}^d$, where $x,y\in\mathbb{R}^d$, we denote the Laplace transform of $Z_n$ by
\begin{equation}\tilde{Z}_n(z)=\int e^{\langle z,\;\omega\rangle}Z_n(d\omega)=\underset{u\in\mathbb{T}_n}\sum e^{\langle z,S_u\rangle}.\end{equation}
Put
\begin{equation}
P_0(z)=1, \qquad P_n(z)=\mathbb{E}_\xi \tilde{Z}_n(z)=\prod^{n-1}_{i=0}m_i(z),
\end{equation}
\begin{equation*}
\tilde{X}_u(z)=\frac{e^{\langle z,S_u\rangle}}{P_{|u|}(z)}
\end{equation*}
and
\begin{equation}
W_n(z)=\frac{\tilde{Z}_n(z)}{P_n(z)}=\sum_{u\in\mathbb{T}_n}\tilde{X}_u(z).
\end{equation}
Let
$$\mathcal{F}_0=\sigma(\xi),\quad
\mathcal{F}_n = \sigma ( \xi, (X(u): |u| <n)) \;\;\text{for $n\geq 1$}$$
 be the $\sigma$-field containing all the
information concerning the first $n$ generations.
It is not difficult to verify that for each $z$ fixed, $W_n(z)$ forms a complex martingale with respect to the filtration $\mathcal{F}_n$ under both laws $\mathbb{P}_\xi$ and $\mathbb{P}$. The convergence of $W_n(z)$ is always useful for  studying the asymptotic properties of $\tilde{Z}_n$ and $Z_n$. In the deterministic environment case, this martingale has been studied by Kahane and  Peyri\`ere \cite{kahane}, Biggins \cite{biggins77,b92}, Durrett and Liggett \cite{durrett}, Guivarc'h \cite{guivarch}, Lyons \cite{lyons} and Liu \cite{liu97, liu00,liu01}, etc. in different contexts.
In particular, for $t\in\mathbb R^d$,  $W_n(t)$ is nonnegative and hence  converges a.s. to a limit random variable $W(t)$ with $\mathbb E_\xi W(t)\leq 1$. In order to study large and moderate deviations associated to $Z_n$, we need the uniform convergence of $W_n(z)$, especially for $z$ in a neighbourhood of $t\in\mathbb R^d$. About this topic, Biggins \cite{b92} have found the uniform  convergence region of  the complex martingale $W_n(z)$ for   branching random walks in deterministic environments, and Wang and Huang \cite{huang17} showed   similar results  for the non-negative martingale  $W_n(t)$  for BRWRE in $\mathbb R$. The following is our conclusion.

\begin{thm}%[Uniform convergence]
\label{Conver1.1}
Let $K$ be a compact subset of $\Omega$ and $K(\delta)=\{z\in\mathbb C^d:|z^j-t^j|<\delta, \forall j=1,\cdots,d,\;\; t\in K\}$ ($\delta>0$) be the neighbourhood of $K$ in space $\mathbb C^d$. Then there exist constants $\delta>0$ and $p_K>1$ such that the complex martingale $W_n(z)$ converges uniformly to a limit random variable $W(z)$ almost surely (a.s.) and in $\mathbb P_\xi$-$L^p$ on $K(\delta)$ for $p\in(1,p_K]$.%, namely,
      %$$\lim_{n\rightarrow\infty}\sup_{z\in K(\delta)}|W_n(z)-W(z)|=0\quad a.s.,$$
    %$$\lim_{n\rightarrow\infty}\sup_{z\in K(\delta)}\mathbb E_\xi|W_n(z)-W(z)|^p=0\quad a.s..$$
\end{thm}

We shall prove Theorem \ref{Conver1.1} with the method of Biggins \cite{b92} (also see Attia \cite{na}). The basic technique is to use Cauchy's formula and the inequality for martingales. %The following lemma is deduced from Cauchy's formula.
%\begin{lem}[\cite{b92}, Lemma 3]\label{LDMAl1}
%Let $z_0=x_0+\mathbf{i}y_0\in\mathbb C^d$ and $\varepsilon>0$.
%If the function $f(z)$ is analytic on $D(z_0,2\varepsilon')$ with $\varepsilon'>\varepsilon$, then
%$$\sup_{z\in D(z_0,\varepsilon)}|f(z)|\leq (2\pi\varepsilon)^{-d}\int_{\partial D(z_0, 2\varepsilon)}|f(\zeta)|d\zeta.$$
%\end{lem}
%\begin{lem}[\cite{b92}, Lemma 1]\label{LDMAl2}
%If $\{X_i\}$ are independent complex random variable with $\mathbb E X_i=0$ or more generally, martingale differences,
%then $\mathbb E|\sum X_i|^p\leq 2^p\sum\mathbb E|X_i|^p$ for $1\leq p\leq 2$.
%\end{lem}
For our model, comparing with the deterministic environment case, the crucial differences in the proof are reflected in the following lemmas.

\begin{lem}\label{LDMAl3}
Let $z_0=x_0+\mathbf{i}y_0\in\mathbb C^d$ and $\varepsilon>0$. Put $W^*(z_0,\varepsilon)=\underset{z\in D(z_0,\varepsilon)}\sup|W_1(z)|$. Then
$$W^*(z_0,\varepsilon)\leq \alpha_0(z_0, \varepsilon)^{-1}\sum_{s\in \mathcal S}\tilde Z_1(s),$$
where $\mathcal S=\{s=(s^1,\cdots,s^d)\in\mathbb R^d:s^j=s_1^j \text{ or } s_2^j,j=1,\cdots,d\}$ with $s_1^j:=x_0^j-\varepsilon$ and $s_2^j:=x_0^j+\varepsilon$.
\end{lem}

\begin{proof}For $z\in  D(z_0,\varepsilon)$, we have $|x^j-x_0^j|<\varepsilon$, so that $s_1^j<x^j<s_2^j$  (for all $j$).
Thus for all $z\in D(z_0,\varepsilon)$,
\begin{eqnarray}\label{LDMAl3e}
|\tilde Z_1(z)|\leq\sum_{u\in\mathbb T_1}e^{\langle x, S_u\rangle}
&=&\sum_{u\in\mathbb T_1}e^{x^1S_u^1}\mathbf{1}_{\{S_u^1<0\}}\exp{\{\sum\limits_{j=2}^dx^jS_u^j\}}+\sum_{u\in\mathbb T_1}e^{x^1S_u^1}\mathbf{1}_{\{S_u^1\geq0\}}\exp{\{\sum\limits_{j=2}^dx^jS_u^j\}}\nonumber\\
&\leq&\sum_{u\in\mathbb T_1}e^{s_1^1S_u^1}\exp{\{\sum\limits_{j=2}^dx^jS_u^j\}}+\sum_{u\in\mathbb T_1}e^{s_2^1S_u^1}\exp{\{\sum\limits_{j=2}^dx^jS_u^j\}}\nonumber\\
&\cdots&\nonumber\\
&\leq&\sum_{s\in \mathcal S}\sum_{u\in\mathbb T_1}e^{\langle s, S_u\rangle}=\sum_{s\in \mathcal S}\tilde Z_1(s).
\end{eqnarray}
Therefore,
$$W^*(z_0,\varepsilon)=\sup_{z\in  D(z_0,\varepsilon)}\left|\frac{\tilde Z_1(z)}{m_0(z)}\right|\leq\frac{\sum_{s\in \mathcal S}\tilde Z_1(s)}{\alpha_0(z_0,\varepsilon)}.$$
\end{proof}

\begin{lem}\label{LDMAl4}
If $t_0\in\Omega_1\bigcap \Omega_2$, then there exist $\varepsilon_0>0$ and $p_0>1$ such that for all $ 0<\varepsilon\leq\varepsilon_0$ and $1<p\leq p_0$,
\begin{equation}\label{LDMAl4e}
\mathbb E\log^+\sup_{z\in  D(t_0,\varepsilon)}\left(\mathbb E_\xi|W_1(z)|^p\right)^{1/p}<\infty.
\end{equation}
\end{lem}

\begin{proof}Since $t_0\in\Omega_2$, there exists $\delta>0$ such that $\mathbb E\log^-\alpha_0(t_0,\delta)<\infty$. Notice that $\alpha_0(t_0,\delta)\leq \alpha_0(t_0,\varepsilon )$ if $\varepsilon \leq \delta$, so that $\mathbb E\log^-\alpha_0(t_0,\varepsilon)\leq\mathbb E\log^-\alpha_0(t_0,\delta)<\infty$. Put $\mathcal{S}_{\varepsilon_0}=\{s=(s^1,\cdots,s^d)\in\mathbb R^d:s^j=t_0^j+\varepsilon_0 \text{ or } t_0^j-\varepsilon_0 ,j=1,\cdots,d\}$. Since $t_0\in \Omega_1$, we can choose $0<\varepsilon_0\leq\delta$ small enough such that $\mathcal{S}_{\varepsilon_0}\subset\Omega_1$. Thus for every $s\in \mathcal{S}_{\varepsilon_0}$, there exists $p_s>1$ such that $\mathbb{E}\log^+\mathbb{E}_\xi\tilde{Z}_1(s)^{p_s}<\infty$. Take $p_0=\underset{\mathcal{S}}\min\{p_s\}$. Using H\"older's inequality, Lemma \ref{LDMAl3} and Minkowski's inequality, we obtain for all $ 0<\varepsilon\leq\varepsilon_0$ and $1<p\leq p_0$,
\begin{eqnarray*}
\sup_{z\in  D(t_0,\varepsilon)}\left(\mathbb E_\xi|W_1(z)|^p\right)^{1/p}&\leq&\sup_{z\in  D(t_0,\varepsilon_0)}\left(\mathbb E_\xi|W_1(z)|^{p_0}\right)^{1/{p_0}}\\
&\leq& \left(\mathbb E_\xi W^*(t_0,\varepsilon_0)^{p_0}\right)^{1/{p_0}}\\
&\leq&\alpha_0(t_0,\varepsilon_0)^{-1}\sum_{s\in \mathcal S_{\varepsilon_0}}\left(\mathbb E_\xi\tilde Z_1(s)^{p_0}\right)^{1/{p_0}},
\end{eqnarray*}
which implies (\ref{LDMAl4e}) immediately since $\mathbb E\log^-\alpha_0(t_0,\varepsilon_0)<\infty$ and $\mathbb{E}\log^+\mathbb{E}_\xi\tilde{Z}_1(s)^{p_0}<\infty$.
\end{proof}

\begin{lem}\label{LDMAl5}
If $t_0\in I\bigcap \Omega_2$, then there exists $p_0>1$ such that for all  $1<p\leq p_0$,
\begin{equation}\label{LDMAl5e}
\lim_{\varepsilon\downarrow0}\mathbb E\log\sup_{z\in  D(t_0,\varepsilon)} \frac{m_0(px)}{|m_0(z)|^p}<0.
\end{equation}
\end{lem}

\begin{proof}
For $z=x+\mathbf{i}y\in\mathbb C^d$ and $p\in\mathbb R$, set  $f(z,p)=\log m_0(px)-p\log |m_0(z)|$ and $g(z,p)=\mathbb E f(z,p)$. While $t_0$ is fixed,
taking the derivative of the function $g(t_0,p)$ with respective to $p$ gives
$$\left.\frac{\partial g(t_0,p)}{\partial p}\right|_{p=1}=\langle t_0,\nabla\Lambda(t_0)\rangle-\Lambda(t_0)<0,$$
since $t_0\in I$. Therefore, as a function varying with $p$, $g(t_0,p)$ is strictly decreasing near $p=1$, hence there exists $p_1>1$ such that
$g(t_0,p)<g(t_0,1)=0$ for $1<p\leq p_1$.

Notice that
\begin{equation}\label{LDMAl5e2}
\mathbb E\log\sup_{z\in  D(t_0,\varepsilon)} \frac{m_0(px)}{|m_0(z)|^p}=
\mathbb E\sup_{z\in  D(t_0,\varepsilon)}\log \frac{m_0(px) }{|m_0(z)|^p}
=\mathbb E\sup_{z\in  D(t_0,\varepsilon)}f(z,p).
\end{equation}
We shall prove that there exist $\varepsilon_1>0$ and $p_2>1$ such that for $1<p\leq p_2$,
\begin{equation}\label{LDMAl5e3}
\mathbb E\sup_{z\in  D(t_0,\varepsilon_1)}|f(z,p)|<\infty.
\end{equation}
Since $t_0\in\Omega_2$, there exists $\delta>0$ such that $\mathbb E\log^-\alpha_0(t_0,\delta)<\infty$.
Take $0<\varepsilon_1<\delta$. Then $D(t_0,\varepsilon_1)\subset D(t_0,\delta)$, so that $|m_0(z)|\geq \alpha_0(t_0,\delta)$ for $z\in D(t_0,\varepsilon_1)$. Moreover, we can choose $p_2>1$ such that $px\in D(t_0,\delta)$ for $z=x+\mathbf i y\in D(t_0,\varepsilon_1)$ if $1<p\leq p_2$.
In fact, for $1<p\leq p_2$, it is clear that
$$|px^i-t_0^i|\leq p|x^i-t_0^i|+(p-1)|t_0^i|\leq p_2\varepsilon_1+(p_2-1)|t_0^i|<\delta$$
if we take $p_2>1$ small enough. %So we have $m_0(px)\geq \alpha_0(t_0,\delta)$ for $z\in D(t_0,\varepsilon_1)$.
On the other hand, by taking expectation $\mathbb E_\xi$ in (\ref{LDMAl3e}), we get for  $z=x+\mathbf i y\in D(t_0,\varepsilon_1)$ and $1<p\leq p_2$,
$$\max\{m_0(x),m_0(px)\}\leq \sum_{s\in \mathcal {S}_{\delta}}m_0(s),$$
where $\mathcal S_{\delta}=\{s=(s^1,\cdots,s^d)\in\mathbb R^d:s^j=t_0^j+\delta \text{ or } t_0^j-\delta ,j=1,\cdots,d\}$. Thus for $1<p\leq p_2$,
\begin{eqnarray*}
\sup_{z\in D(t_0,\varepsilon_1)}|f(z,p)|&\leq& \sup_{z\in D(t_0,\varepsilon_1)}\left(\log^+m_0(px)+\log^-m_0(px)+p\log^+m_0(x)+p\log^-|m_0(z)|\right)\\
&\leq& (p+1)\left(\log^+\sum_{s\in \mathcal {S}_{\delta}}m_0(s)+\log^- \alpha_0(t_0,\delta)\right).
\end{eqnarray*}
Taking expectation $\mathbb E$ in the above inequality  yields (\ref{LDMAl5e3}), since $\mathbb E\log^-\alpha_0(t_0,\delta)<\infty$ and $\Lambda(s)$ is a  real number for every $s$.

Take $p_0=\min\{p_1,p_2\}$. For $1<p\leq p_0$, letting $\varepsilon\downarrow0$ in (\ref{LDMAl5e2}) and noticing (\ref{LDMAl5e3}), by the dominated convergence theorem, we deduce
$$\lim_{\varepsilon\downarrow0}\mathbb E\log\sup_{z\in  D(t_0,\varepsilon)} \frac{m_0(px)}{|m_0(z)|^p}
=\lim_{\varepsilon\downarrow0}\mathbb E\sup_{z\in  D(t_0,\varepsilon)}f(z,p)=\mathbb Ef(t_0,p)=g(t_0,p)<0.$$
The proof is complete.
\end{proof}

\begin{lem}[\cite{wh19}, Lemma 2.1]\label{LMDS5L2}
Let $(\alpha_n,\beta_n,\gamma_n)_{n\geq0}$ be a stationary and ergodic sequence of non-negative random variables. If $\mathbb{E}|\log\alpha_0|<\infty$, $\mathbb{E}|\log\beta_0|<\infty$ and  $\mathbb{E}\log^+\gamma_0<\infty$, then
\begin{equation}\label{LMDS5L2e1}
\limsup_{n\rightarrow\infty}\frac{1}{n}\log\left(\gamma_n\sum_{k=0}^{n}\alpha_0\cdots\alpha_{k-1}\beta_k\cdots\beta_{n-1}\right)\leq\max\{\mathbb E\log\alpha_0,\mathbb{E}\log\beta_0\}\quad a.s..
\end{equation}
\end{lem}

From Lemma \ref{LMDS5L2}, we  deduce the following statements.

\begin{co}\label{LPBRWL2.1}
Let $(\alpha_n,\beta_n,\gamma_n)_{n\geq0}$ be a stationary and ergodic sequence of non-negative random variables. Assume that  $\mathbb{E}|\log\alpha_0|<\infty$, $\mathbb{E}|\log\beta_0|<\infty$ and  $\mathbb{E}\log^+\gamma_0<\infty$.
\begin{itemize}
\item[(a)]If $\max\{\mathbb{E}\log\alpha_0, \mathbb{E}\log\beta_0\}<0$, then the series
%\begin{equation}\label{CRE2.2.1a}
$\sum\limits_{n=0}^{\infty}\gamma_n\sum\limits_{k=0}^{n}\alpha_0\cdots\alpha_{k-1}\beta_k\cdots\beta_{n-1} <\infty\quad a.s..$
%\end{equation}

\item[(b)]If $\mathbb{E}\log\alpha_0<0$, then the series
%\begin{equation}\label{CRE2.2.1}
$\sum\limits_{n=0}^{\infty}\alpha_0\cdots\alpha_{n-1}\gamma_n<\infty\quad a.s..$
%\end{equation}
\end{itemize}
\end{co}

\medskip

%%%%%%%%%%%%%%%%%%%%%%%%%   section 3   %%%%%%%%%%%%%%%%%%%%%%%%%%%%%%%%%%%%%%
%\section{Proof of Theorem \ref{Conver1.1}}\label{LDMDS3}
 Now we can state the proof of Theorem \ref{Conver1.1}.
%\subsection{Uniform convergence}
\begin{proof}[Proof of Theorem \ref{Conver1.1}]
As the function $W_n(z)$ is analytic on $\mathbb C^d$, using Cauchy's formula like  \cite{b92}, it suffices to  show that for each $t_0\in K$, there exists $D(t_0,\varepsilon)\in\mathbb C^d$ ($\varepsilon>0$ small enough) such that
\begin{equation}\label{LDMAS3pe3}
\sum_n \sup_{z\in D(t_0,\varepsilon)}\left(\mathbb E_\xi|W_{n+1}(z)-W_n(z)|^p\right)^{1/p}<\infty\quad a.s..
\end{equation}
for suitable $p\in(1,2]$.
Observe that
$$W_{n+1}(z)-W_n(z)=\sum_{u\in\mathbb T_n}\tilde {X}_u(z)(W_1(u,z)-1),$$
where under the quenched law $\mathbb P_\xi$, $\{W_{k}(u,z)\}_{|u|=n}$ are i.i.d. and independent of $\mathcal F_n$ with common distribution determined by $\mathbb P_\xi(W_{k}(u,z)\in \cdot)=\mathbb P_{T^n\xi}(W_k(z)\in\cdot)$ (the precise definition of $W_k(u,z)$ will be given by (\ref{LDMDWE}) later). The notation $T$ represents the shift operator: $T^n\xi=(\xi_n,\xi_{n+1},\cdots)$ if $\xi=(\xi_0,\xi_1,\cdots)$. Applying Burkholder's inequality  to $\{\tilde {X}_u(z)(W_1(u,z)-1)\}$, we obtain
\begin{equation}\label{LDMAS3pe3}
\mathbb E_\xi|W_{n+1}(z)-W_n(z)|^p\leq C \mathbb E_\xi\sum_{u\in\mathbb T_n}|\tilde {X}_u(z)|^p|W_1(u,z)-1|^p
\leq C \frac{P_n(px)}{|P_n(z)|^p}\mathbb E_{T^n\xi}|W_1(z)-1|^p,
\end{equation}
where $C>0$ is a constant independent of $\xi$ and $n$, and in general it does not necessarily stand for the same constant throughout.
Put
\begin{equation}\label{LDMAS3pt1ee}
\beta_n=\sup_{z\in D(t_0,\varepsilon)}\frac{m_n(px)^{1/p}}{|m_n(z)|}\quad\text{and}\quad\gamma_n=\sup_{z\in D(t_0,\varepsilon)} \left(\mathbb E_{T^n\xi}|W_1(z)-1|^p\right)^{1/p}.
\end{equation}
By Lemma \ref{LDMAl4}, there exist $\varepsilon_1>0$ and $p_1>1$ such that (\ref{LDMAl4e}) holds for all $ 0<\varepsilon\leq\varepsilon_1$ and $1<p\leq p_1$. Meanwhile, Lemma \ref{LDMAl5} implies that there exists $p_2>1$ such that (\ref{LDMAl5e}) holds for all $1<p\leq p_2$.
Take $p\leq\min\{p_1,p_2\}$ and fix this $p$. By (\ref{LDMAl5e}), there exists $\varepsilon_p>0$ such that for all $0<\varepsilon\leq \varepsilon_p$,
$$\mathbb E\log\beta_0=\frac{1}{p}\mathbb E\log\sup_{z\in  D(t_0,\varepsilon)} \frac{m_0(px)}{|m_0(z)|^p}<0.$$
Take $0<\varepsilon\leq \min\{\varepsilon_1,\varepsilon_p\}$. So we have $\mathbb E\log\beta_0<0$ and $\mathbb E\log^+\gamma_0<\infty$ (by (\ref{LDMAl4e})). Then it follows from (\ref{LDMAS3pe3}) and Corollary \ref{LPBRWL2.1}(b) that
$$\sum_n \sup_{z\in D(t_0,\varepsilon)}\left(\mathbb E_\xi|W_{n+1}(z)-W_n(z)|^p\right)^{1/p}
\leq \sum_n\beta_0\cdots\beta_{n-1}\gamma_n<\infty\quad a.s.,$$
which completes the proof.
\end{proof}

\begin{re}%Denote the limit of $W_n(z)$ (if it exists)  by
%\begin{equation}
%W(z)=\lim_{n\rightarrow\infty}W_n(z)\qquad a.s..
%\end{equation}
Theorem \ref{Conver1.1} says that $W_n(z)$ converges to $W(z)$ uniformly a.s. and in $\mathbb P_\xi$-$L^p$ for some $p\in(1,2]$
on  set $K(\delta)$. Noticing (\ref{LDMAS3pe3}), we deduce that for every $z\in K(\delta)$,
\begin{eqnarray}\label{LDMDS3re}
(\mathbb E_\xi|W(z)|^p)^{1/p}%&=& \lim_{n\rightarrow\infty}(\mathbb E_\xi|W_n(z)|^p)^{1/p}\nonumber\\
%&\leq& \lim_{n\rightarrow\infty}(\mathbb E_\xi|W_n(z)-1|^p)^{1/p}+1\nonumber\\
&\leq& \lim_{n\rightarrow\infty}\sum_{k=0}^{n-1}(\mathbb E_\xi|W_{k+1}(z)-W_k(z)|^p)^{1/p}+1\nonumber\\
&\leq&  C\sum_{k=0}^{\infty} \frac{P_k(px)^{1/p}}{|P_k(z)|}\left(\mathbb E_{T^k\xi}|W_1(z)-1|^p\right)^{1/p}+1.
\end{eqnarray}
\end{re}

Theorem \ref{Conver1.1} is principally  concerned with the uniform convergence of $W_n(z)$ near the real number. From similar arguments to the proof of Theorem \ref{Conver1.1}, we can obtain a general uniform convergence region for the complex martingale $W_n(z)$,  see  Theorem \ref{Conver1.1cc} below  whose proof is omitted. Set
$$I_p'=\{z\in\mathbb{C}^d:\Lambda(px)- p\mathbb E\log |m_0(z)|<0\},$$
$$\Omega_{1,p}'=\mbox{int}\{z\in\mathbb{C}^d:\mathbb{E}\log^+\mathbb{E}_\xi\tilde{Z}_1(x)^p<\infty\},$$
\begin{equation*}
\Omega_2'=\{z\in\mathbb{C}^d:\exists\delta_z>0\;\;\text{such that}\;\;
\mathbb{E}\log^-\alpha_0(z,\delta_z)<\infty\},
\end{equation*}
$$ \Omega_p'=I_p'\bigcap\Omega'_{1,p}\quad\text{and}\quad \Omega'=\left(\bigcup_{1<p\leq2}\Omega_p'\right)\bigcap\Omega_2'. $$

\begin{thm}\label{Conver1.1cc}Let $K$ be a compact subset of $\Omega'$. Then there exists constant  $p_K>1$ such that the complex martingale $W_n(z)$ converges uniformly to a limit random variable $W(z)$ almost surely (a.s.) and in $\mathbb P_\xi$-$L^p$ on $K$ for $p\in(1,p_K]$.
\end{thm}

 As the complication of  Theorem \ref{Conver1.1}, Theorem \ref{Conver1.1cc} generalizes the result of Biggins (\cite{b92}, Theorem 2) for classical branching random walks. It can be seen that $\Omega$ is actually the intersection of $\Omega'$ with the real space, i.e. $\Omega=\Omega'\bigcap\mathbb R^d$, which means that Theorem \ref{Conver1.1} is in fact contained in  Theorem \ref{Conver1.1cc}.

%%%%%%%%%%%%%%%%%%%%%%%%%%%%%%%%%%%%%%%%%%%%%%%%%%%%%%%%%%%%%%%%%%%%%%%%%%%%%%%%%%%%%%%%%%%%%%%%%%%%%%%%%%%%%%%%%%%%%%%%%%%%%%%%%%%%%%%%%%%%%%
\section{Multifractal analysis and  proof of Theorem \ref{Conver1.12}}\label{LDMDS3}
%\subsection{Results of multifractal analysis and proof of Theorem \ref{Conver1.12}}
The proof of Theorem \ref{Conver1.12} is based on   multifractal analysis.
For the Galton-Watson tree $\mathbb T$, the boundary of $\mathbb T$ is defined as
$$\partial \mathbb T=\{u\in \mathbb I: u|n\in \mathbb T \text{ for all }n\in \mathbb N\}.$$
As a subset of $\mathbb I={\mathbb N^*}^{\mathbb N^*}$, $\partial \mathbb T$ is a metrical and compact topological space with
$$[u]=\{w\in\partial \mathbb T: w||u|=u\},\qquad u\in\mathbb T$$
its topological basis, and with the standard ultrametric distance
$$d(u,v)=e^{-\sup{\{|w|:\; u,v\in [w]\}}}.$$
Confirming the Hausdorff dimensions of  the level sets
\begin{equation}
E(\alpha)=\{u\in\partial\mathbb T:\lim_{n\rightarrow\infty}\frac{S_{u|n}}{n}=\alpha\}  \quad( \alpha\in\mathbb R^d)
\end{equation}
 is in the frame work of multifractal analysis.
As usual, the $s$-dimensional Hausdorff measure  of a set $E$ is defined as
$$H^s(E)=\liminf_{\delta\rightarrow0^+}\left\{\sum_{i=1}^\infty (\text{diam}(U_i))^s: E\subset \bigcup_{i=1}^\infty U_i, \text{diam}(U_i)\leq\delta\right\},$$
where the sequence of sets $\{U_i\}$ is a coverage of $E$ in $\partial \mathbb T$ and  $\text{diam}(U_i)$ represents the diameter of $U_i$. Then the  Hausdorff dimensions of $E$ is defined as
$$\dim E=\sup\{s>0: H^s(E)=\infty\}=\inf\{s>0: H^s(E)=0\}.$$
Moreover, we  define  the pressure-like function
\begin{equation}\tilde{\Lambda}(t)=\underset{n\rightarrow\infty}\limsup\frac{1}{n}\log \tilde{Z}_n(t)\quad(t\in\mathbb{R}^d),
\end{equation}
whose   Legendre transform is denoted by
$\tilde\Lambda^*(\alpha)=\underset{t\in\mathbb{R}^d}\sup\left\{\langle t,\alpha\rangle-\tilde\Lambda(t)\right\}$.
For the level sets $E(\alpha)$, we say that the multifractal formalism holds at $\alpha$ if $\dim E(\alpha)=-\tilde\Lambda^*(\alpha)$ (see e.g. \cite{bg92, ol, pj04} for introductions of  multifractal formalisms). For our model, we find the region where the multifractal formalism holds.

\begin{thm}[Multifractal analysis]\label{Conver1.12ma}
 Let $ \mathcal{J}=\{\nabla\Lambda(t)\in\mathbb{R}^d:t\in\Omega\}$. Then with probability 1, for all $\alpha\in\mathcal{J}$, we have $ \dim E(\alpha)=-\tilde\Lambda^*(\alpha)=-\Lambda^*(\alpha)$.
\end{thm}
Theorem \ref{Conver1.12ma} means that the multifractal formalism holds at $\alpha\in\mathcal J$, which is an extension of the result of Attia \cite{na} for a branching random walk in a deterministic environment.  Under an additional condition for  $L_1$,  the  region  $\mathcal J$ can be enlarged to $\mathcal J\bigcup(\mbox{int}\tilde{\mathcal J})$.

\begin{co}\label{Conver1.12cma} Let $ \tilde{\mathcal J}=\{\alpha\in\mathbb{R}^d:\Lambda^*(\alpha)<0\}$. If $\mathbb E[\mathbb P_\xi (\|L_1\|\leq a)^{-1}]<\infty$ for some constant $a>0$, then with probability 1, for all $\alpha\in\mbox{\emph{int}}\tilde{\mathcal J}$, we have $ \dim E(\alpha)=-\tilde\Lambda^*(\alpha)=-\Lambda^*(\alpha)$.%
\end{co}

It is worth mentioning that  for a branching random walk in a deterministic environment, Attia and Barral \cite{na2014} showed that the multifractal formalism holds for all $\alpha\in\{\alpha\in\mathbb{R}^d:\Lambda^*(\alpha)\leq0\}$ (namely, the boundary problem was solved) through tedious approximation and computation. Here, in order to  highlight the treatment of random environments and  avoid a great length of the paper, we shall not discuss such a method  for our model.

\medskip
Using Theorem \ref{Conver1.12ma} and Corollary   \ref{Conver1.12cma}, we can achieve the proof of  Theorem \ref{Conver1.12}.

\begin{proof}[Proof of Theorem \ref{Conver1.12}]
By the general upper bounds in the theory of large deviations (see (\cite{z}, Theorem 4.5.3)) and Proposition \ref{LDMDS3l1},  we have  with probability $1$,
for all measurable sets $A\subset\mathbb R^d$,
\begin{equation}\label{LDMDP3}
\limsup_{n\rightarrow\infty}\frac{1}{n}\log Z_n(nA)\leq
-\inf_{\alpha\in  \bar A}\tilde\Lambda^*(\alpha)\leq-\inf_{\alpha\in  \bar A}\Lambda^*(\alpha).
\end{equation}
For the lower bounds, denote
 $\underline{Z}=\liminf_{n}\frac{1}{n}\log Z_n(nA)$.  Since $\Lambda^*(\alpha)\geq -\Lambda(0)=-\mathbb E\log m_0(0)$ by the definition of $\Lambda^*$,
 from (\ref{LDMDP3}) we can see that $$\underline{Z}\leq\limsup_{n\rightarrow\infty}\frac{1}{n}\log Z_n(nA)\leq\mathbb E\log m_0(0)<\infty \qquad a.s.. $$
 For $ \alpha\in\mbox{int}   A$, take $\varepsilon>0$ small enough such that $B(\alpha,\varepsilon)\subset A$.
For $u\in E(\alpha)$,   we have $\|\frac{S_{u|n}}{n}-\alpha\|\leq\varepsilon$ for $n$ large enough,
which means that $E(\alpha)\subset\underset{u\in E_n}\bigcup[u]$, where $E_n=\{u\in\mathbb T_n: S_u\in n B(\alpha,\varepsilon)\}$. Obviously, the  diameter of the set $[u]$ for $u\in E_n$ is less than $e^{-n}$, and so is less than $\delta>0$ for $n$ large enough. Therefore, the $s$-dimension Hausdorff measure of $E(\alpha)$ satisfies
$$\mathcal{H}^s(E(\alpha))%=\lim_{\delta\rightarrow 0^+}\mathcal{H}^s_\delta(E(\alpha))
\leq \liminf_{n\rightarrow\infty}e^{-ns}Z_n(nB(\alpha,\varepsilon))\leq \liminf_{n\rightarrow\infty}e^{-ns}Z_n(nA)=0$$
provided $s>\underline{Z}$, which implies that $\dim E(\alpha)\leq s$ and furthermore $\dim E(\alpha)\leq\underline{Z}$ by the arbitrary of $s$. Taking superior on $\alpha$, we get  $\sup\limits_{\alpha\in\mbox{int}   A}\dim E(\alpha)\leq \underline{Z}$. Combing this result with  Theorem \ref{Conver1.12ma} (resp. Corollary   \ref{Conver1.12cma}) yields
$-\inf\limits_{\alpha\in\mbox{int}   A}\Lambda^*(\alpha)=\sup\limits_{\alpha\in\mbox{int}   A}\dim E(\alpha)\leq \underline{Z}$ if $\mbox{int} A\subset \mathcal J$ (resp. $\mbox{int} A\subset\mbox{int}  \tilde{\mathcal J}$).
\end{proof}

\section{Proof of Theorem \ref{Conver1.12ma}}\label{LDMDS31}
We shall prove  Theorem \ref{Conver1.12ma} through the natural approach of finding the upper and lower bounds for the Hausdorff dimensions of the level sets $E(\alpha)$,   according to  the ideas showed in Attia \cite{na}. Firstly, the upper bounds  can be deduced from the two propositions below.

\begin{pr}[\cite{na}, Proposition 2.2]\label{LDMDSSpp2}
With probability $1$, for all $\alpha\in\mathbb{R}^d$, $\dim E(\alpha)\leq-\tilde{\Lambda}^*(\alpha)$, where $\dim E(\alpha)<0$ means that $E(\alpha)$ is empty.
\end{pr}

\begin{pr}\label{LDMDS3l1}
With probability $1$, $\tilde{\Lambda}(t)\leq\Lambda(t)$ for all  $t\in\mathbb{R}^d$, and then $\tilde{\Lambda}^*(\alpha)\geq\Lambda^*(\alpha)$ for all $ t\in\mathbb{R}^d$.
\end{pr}
\begin{proof}
Since the functions $\tilde{\Lambda}(t)$ and $\Lambda(t)$ are convex and thus continuous, we only need to prove  $\tilde{\Lambda}(t)\leq\Lambda(t)$ a.s. for each  $t\in\mathbb{R}^d$. Fix $t\in\mathbb{R}^d$.
For $s>\Lambda(t)$, we have $\mathbb{E}\log(e^{-s}m_0(t))=\Lambda(t)-s<0$. Thus by Corollary \ref{LPBRWL2.1}(b),
$$\mathbb{E}_\xi\sum_{n} e^{-ns}\tilde{Z}_n(t)=\sum_{n}e^{-ns}P_n(t)<\infty\quad a.s.,$$
which implies that
$\sum_{n}e^{-ns}\tilde{Z}_n(t)<\infty$ a.s..
Hence, $\tilde{Z}_n(t)=O(e^{ns})$ a.s., which yields $\tilde{\Lambda}(t)\leq s$ a.s..
Since $s>\Lambda(t)$ is arbitrary, we get the conclusion.
\end{proof}

\medskip
In order to obtain the lower bounds, we need  rely on the associated Mandelbrot measure \cite{man74}.
Let $\mathbb T(u)$ be the  Galton-Watson tree rooted at $u\in\mathbb T$ and
$\mathbb T_n(u)=\{uv\in\mathbb T: |v|=n\}$ be the set of particles in the $n$-th generation of  $\mathbb T(u)$. For $u\in\mathbb T$ and $z\in \mathbb C^d$, denote
\begin{equation}
\tilde Z_n(u,z)=\sum_{uv\in\mathbb T_n(u)}e^{\langle z,S_{uv}-S_u\rangle}
\end{equation}
and
\begin{equation}\label{LDMDWE}
W_n(u,z)=\frac{\tilde Z_n(u,z)}{\mathbb E_\xi\tilde Z_n(u,z)}=\frac{\sum_{uv\in\mathbb T_n(u)}e^{\langle z,S_{uv}-S_u\rangle}}{m_{|u|}(z)\cdots m_{|u|+n-1}(z)}.
\end{equation}
In particular, we have $\tilde Z_n(\emptyset,z)=\tilde Z_n(z)$ and $W_n(\emptyset, z)=W_n(z)$. % and $W(\emptyset,t)=W(t)$.
For $t\in\mathbb R^d$, the martingale $\{W_n(u,t)\}$ is non-negative, and hence it converges a.s. to a limit
\begin{equation}
W(u,t):=\lim_{n\rightarrow\infty}W_n(u,t).
\end{equation}
By the branching property, we can see that
$$W(u,t)=\sum_{ui\in\mathbb T_1(u)}\frac{e^{\langle t,L_i(u)\rangle}}{m_{|u|}(t)}W(ui,t).$$
Thus for each $t\in\mathbb R^d$, we can define a unique measure $\mu_t$ on $\partial \mathbb T$ such that
\begin{equation}
\mu_t([u])=\tilde X_u(t)W(u,t).
\end{equation}
This measure  $\mu_t$ is  the so-called  \emph{ Mandelbrot measure} for BRWRE. %\marginpar{ref????}Obviously,
Clearly, $\mu_t$ is finite with $\mu_t(\partial \mathbb T)=W(t)$.

\begin{thm}\label{LS1.4TT1}
With probability $1$, for all $t\in\Omega$,
$$\lim_{n\rightarrow\infty}\frac{1}{n}\log \mu_t([u|n])= \langle t,\nabla\Lambda(t)\rangle-\Lambda(t)\qquad \text {for $\mu_t$-a.s. $u\in \partial \mathbb T$.}$$
\end{thm}

Theorem \ref{LS1.4TT1} generalizes the result of Liu and Rouault \cite{liu96} for the Galton-Watson processes, and that of Attia \cite{na} for classical branching random walks. Such a result
allows the calculations of the dimension of $\mu_t$, as well as the Hausdorff and Packing dimensions of the support of $\mu_t$ and those of the level sets $E(\alpha)$, just as what were done in \cite{kahane, kahane87, liu00, na,na2014}.

By the definition of $\mu_t$, we see that
\begin{equation}\label{LDMDS4em}
\frac{1}{n}\log\mu_t([u|n])=\langle t, \frac{S_{u|n}}{n}\rangle-\frac{1}{n}\log P_n(t)+\frac{1}{n}\log W(u|n,t).
\end{equation}
The ergodic theorem gives that $\frac{1}{n}\log P_n(t)\rightarrow\Lambda (t)$ a.s. as $n\rightarrow\infty$. To prove Theorem \ref{LS1.4TT1}, we need to calculate the limits of the other two terms in the right hand side of (\ref{LDMDS4em}). To this end, the following lemma is useful.

\begin{lem}\label{LDMAl7}
Let $(\alpha_n)_{n\geq0}$ be a stationary and ergodic sequence of non-negative random variables satisfying $\mathbb E\log \alpha_0\in(-\infty,0)$.
If $t_0\in\Omega$, then there exist $\varepsilon_0>0$ and $p_0\in(1,2]$ such that for all $ 0<\varepsilon\leq\varepsilon_0$ and $1\leq p\leq p_0$, the series
\begin{equation}\label{LDMDS3l2}
\sum_n \alpha_0\cdots\alpha_{n-1}\sup_{z\in D(t_0,\varepsilon)}\left(\mathbb E_{T^n\xi}|W(z)|^p\right)^{1/p}<\infty\quad a.s..
\end{equation}
\end{lem}

\begin{proof}
Recall $\beta_n$ and $\gamma_n$ defined in (\ref{LDMAS3pt1ee}).  According to the proof of Theorem \ref{Conver1.1}, we can take $\varepsilon_0>0$ and $p_0>1$ such that $\mathbb E\log \beta_0<0$
and $\mathbb E\log^+\gamma_0<\infty$ with $\varepsilon=\varepsilon_0$ and $p=p_0$.
By (\ref{LDMDS3re}),  we    deduce that for all  $0<\varepsilon\leq\varepsilon_0$ and $1\leq p\leq\min\{2,p_0\}$,
$$\sup_{z\in D(t_0,\varepsilon)}\left(\mathbb E_{T^n\xi}|W(z)|^p\right)^{1/p}\leq\sup_{z\in D(t_0,\varepsilon_0)}\left(\mathbb E_{T^n\xi}|W(z)|^{p_0}\right)^{1/{p_0}}\leq C\sum_{k=n}^\infty\beta_n\cdots\beta_{k-1}\gamma_k+1.$$
Thus
\begin{eqnarray*}
\sum_{n=0}^\infty\alpha_0\cdots\alpha_{n-1}\sup_{z\in D(t_0,\varepsilon)}\left(\mathbb E_{T^n\xi}|W(z)|^p\right)^{1/p}
&\leq&C\sum_{n=0}^\infty \alpha_0\cdots\alpha_{n-1}\sum_{k=n}^\infty\beta_n\cdots\beta_{k-1}\gamma_k+\sum_{n=0}^\infty \alpha_0\cdots\alpha_{n-1}.
%\\&=&C\sum_{k=0}^\infty\gamma_k\sum_{n=0}^k a^{-n}\beta_n\cdots\beta_{k-1}+\sum_{n=0}^\infty a^{-n}
\end{eqnarray*}
Corollary \ref{LPBRWL2.1}(b) shows that the second series in the right hand side of the inequality above converges a.s. since $\mathbb E\log\alpha_0<0$. For the first series,
we have
$$\sum_{n=0}^\infty\alpha_0\cdots\alpha_{n-1}\sum_{k=n}^\infty\beta_n\cdots\beta_{k-1}\gamma_k
=\sum_{k=0}^\infty\gamma_k\sum_{n=0}^k\alpha_0\cdots\alpha_{n-1}\beta_n\cdots\beta_{k-1}<\infty\quad a.s.$$
from Corollary \ref{LPBRWL2.1}(a), since $\mathbb E\log^+\gamma_0<\infty$ and $\max\{\mathbb E\log \alpha_0,\mathbb E\log \beta_0\}<0$.

\end{proof}

\begin{pr}\label{LDMDS3p2}
With probability $1$, for all $t\in\Omega$,
$$\lim_{n\rightarrow\infty}\frac{1}{n}\log W(u|n,t)=0\quad\text{
for $\mu_t$-a.s. $u\in\partial \mathbb T$.}$$
\end{pr}
\begin{proof}
Let $a>1$ be a constant.
For  $t\in\mathbb{R}^d$ and $n\geq1$, we set
$E_n^-=\{u\in\partial \mathbb T:W(u|n,t)\leq a^{-n}\}$ and
$
E_n^+=\{u\in\partial \mathbb T:W(u|n,t)\geq a^n\}.
$
%Let $K$ be a compact subset of $\Omega$.
It suffices to show that for $E_n\in\{E_n^-,E_n^+\}$,
%\begin{equation}\label{LDMDS3PPe1}
%\mathbb E_\xi\left(\sup_{t\in K}\sum_n\mu_t(E_n)\right)<\infty\quad a.s.,
%\end{equation}
%which leads to the desired result by Borel-Cantelli lemma. As $K$ is compact, it can be covered by a finite number of sets $B(t_i,\varepsilon_i/2)$.
%Therefore, to show (\ref{LDMDS3PPe1}), we only need to prove that
for each $t_0\in\Omega$, there exists $\varepsilon>0$ small enough such that
\begin{equation}\label{LDMDS3PPe2}
\mathbb E_\xi\left(\sup_{t\in B(t_0,\varepsilon/2)}\sum_n\mu_t(E_n)\right)<\infty\quad a.s..
\end{equation}
We just consider the case where $E_n=E_n^+$. The proof for $E_n=E_n^-$ is similar.
Notice that for $p>1$,
$$\mu_t(E_n^+)=\sum_{u\in \mathbb T_n}\tilde X_u(t)W(u,t)\mathbf{1}_{\{W(u,t)\geq a^{n}\}}\leq  a^{-(p-1)n}\sum_{u\in \mathbb T_n}\tilde X_u(t)W(u,t)^p. $$
Theorem \ref{Conver1.1} ensures that the limit $W(u,z):=\lim_n W_n(u,z)$ exists a.s. and analytic on $D(t_0,\varepsilon)$ for some $\varepsilon>0$.
%\marginpar{the definition of $W_n(u,z)$????}
Applying Cauchy's formula, we obtain
\begin{eqnarray}\label{LDMDS3PPe3}
\sum_n\mathbb E_\xi\left(\sup_{t\in B(t_0,\varepsilon/2)}\mu_t(E_n^+)\right)
&\leq&C\sum_n a^{-(p-1)n}\sup_{z\in D(t_0,\varepsilon)}\mathbb E_\xi \sum_{u\in \mathbb T_n}|\tilde X_u(z)|\left|W(u,z)\right|^p\nonumber \\
%&\leq&C \sum_n a^{-(p-1)n}\sup_{z\in D(t_0,\varepsilon)}\frac{P_n(x)}{|P_n(z)|}\mathbb E_{T^n\xi}|W(z)|^p\nonumber\\
&\leq&C\left(\sum_n\tilde d_0(\varepsilon)\cdots \tilde d_{n-1}(\varepsilon)\left(\mathbb E_{T^n\xi}|W(z)|^p\right)^{1/p}\right)^p,
\end{eqnarray}
where
$$ \tilde d_{n}(\varepsilon)=a^{-\frac{p-1}{p}}\left(\sup_{z\in D(t_0,\varepsilon)}\frac{m_n(x)}{|m_n(z)|}\right)^{1/p}.$$
Similarly to the proof of Lemma \ref{LDMAl5}, the dominated convergence theorem gives
$$\lim_{\varepsilon\downarrow0}\mathbb E\log \tilde d_0(\varepsilon)=-\frac{p-1}{p}\log a<0.$$
Therefore, by Lemma \ref{LDMAl7}, the series in the last line of (\ref{LDMDS3PPe3}) converges a.s. for some $p\in(1,2]$ and $\varepsilon>0$ small enough,
so (\ref{LDMDS3PPe2}) holds for $E_n=E_n^+$.
\end{proof}

\medskip
For $t\in\Omega$ and $n\in\mathbb N$, let us define a measure on $\mathbb{R}^d$ as:
\begin{equation}
\nu_{t,n}(\cdot)=\mu_t\left(u\in\partial \mathbb T:\frac{S_{u|n}}{n}\in\cdot\right).
\end{equation}
For $(t,\lambda)\in \mathbb R^d\times \mathbb R^d$, let
\begin{equation}
L_n(t,\lambda)=\frac{1}{n}\log\int_{\mathbb{R}^d}e^{n\langle\lambda,x\rangle}d\nu_{t,n}(x)
~~\text{and}~~L_t(\lambda)=\limsup_{n\rightarrow\infty}L_n(t,\lambda).
\end{equation}

\begin{pr}\label{LDMDS3p3}
Let $t_0\in\Omega$. Then there exists  $\varepsilon_0>0$,  such that
\begin{equation}\label{LDMDS3p3e1}
\mathbb P(L_t(\lambda)=\Lambda(t+\lambda)-\Lambda(t),\;\;\forall t\in B(t_0,\varepsilon_0),\;\;\forall \lambda\in B(0,\varepsilon_0))=1.
\end{equation}
\end{pr}

\begin{proof}
For $(z,z')\in\mathbb C^d\times \mathbb C^d$, denote
$$V_n(z,z')=\sum_{u\in\mathbb T_n}\tilde X_u(z+z')W(u,z)$$
when $W(u,z):=\lim_n W_n(u,z)$ exists. In particular, $V_n(t,z')$ is always well defined for $(t,z')\in \mathbb R^d\times \mathbb C^d$, as $W(u,t)$ is always exists for all $t\in\mathbb R^d$. Calculate that for $(t,\lambda)\in \mathbb R^d\times \mathbb R^d$,
\begin{eqnarray}\label{LDMDS3p3e2}
L_n(t,\lambda)
%=\frac{1}{n}\log\int_{u\in\partial \mathbb T}e^{n\langle\lambda,\frac{S_{u|n}}{n}\rangle}d\mu_t(u)%\nonumber\\%&=&\frac{1}{n}\log\sum_{u\in \mathbb T_n}e^{\langle\lambda,S_{u}\rangle}\frac{e^{\langle t,S_u\rangle}}{P_n(t)}W(u,t)\nonumber\\
=\frac{1}{n}\log V_n(t,\lambda)+\frac{1}{n}\log P_n(t+\lambda)-\frac{1}{n}\log P_n(t).
\end{eqnarray}
%The ergodic theorem implies that
%$$\lim_{n\rightarrow\infty}\frac{1}{n}\log P_n(t)=\mathbb E\log m_0(t)\qquad a.s..$$
Hence to obtain (\ref{LDMDS3p3e1}), it suffices to show that $V_n(t,\lambda)$ converges uniformly a.s. and in $\mathbb P_\xi$-$L^1$ on $B(t_0,\varepsilon_0)\times B(0,\varepsilon_0)$ for some $\varepsilon_0>0$, which ensures that a.s., the limit
$$V(t,\lambda)=\lim_{n\rightarrow\infty}V_n(t,\lambda)>0\qquad\text{on}\quad   B(t_0,\varepsilon_0)\times B(0,\varepsilon_0),$$
and so (\ref{LDMDS3p3e1}) holds by letting $n\rightarrow\infty$ in (\ref{LDMDS3p3e2}).

Now we work on the uniform convergence (a.s. and in $\mathbb P_\xi$-$L^1$) of  $V_n(t,\lambda)$. By Theorem \ref{Conver1.1}, we see that
there exists $\delta>0$ such that $V_n(z,z')$ is well defined and analytic on $(z,z')\in D(t_0,\delta)\times\mathbb C^d$. Similarly to the proof of
Theorem \ref{Conver1.1}, we can turn to show that
\begin{equation}\label{LDMDS3p3e3}
\sum_n \sup_{\substack {z\in D(t_0,\varepsilon)\\z'\in D(0,\varepsilon)}}\left(\mathbb E_\xi|V_{n+1}(z,z')-V_n(z,z')|^p\right)^{1/p}<\infty\quad a.s.
\end{equation}
for suitable $0<\varepsilon\leq\delta$ and $p\in(1,2]$. Observe that
$$V_{n+1}(z,z')-V_n(z,z')=\sum_{u\in\mathbb T_n}\tilde {X}_u(z+z')\left(\sum_{ui\in\mathbb T_1(u)}\frac{e^{\langle z+z',L_i(u)\rangle}}{m_n(z+z')}W(ui,z)-W(u,z)\right).$$
Using Burkholder's inequality, we deduce
\begin{equation}\label{LDMDS3p3e4}
\mathbb E_\xi|V_{n+1}(z,z')-V_n(z,z')|^p
\leq C \frac{P_n(p(x+x'))}{|P_n(z+z')|^p}\mathbb E_{T^n\xi}|V_1(z,z')-W(z)|^p.
\end{equation}
By Minkowski's inequality, we have
\begin{equation}\label{LDMDS3p3e5}
\left(\mathbb{E}_\xi\left|V_1(z,z')-W(z)\right|^p\right)^{\frac{1}{p}}
\leq\left(\mathbb{E}_\xi\left|V_1(z,z')\right|^p\right)^{\frac{1}{p}}+\left(\mathbb{E}_\xi|W(z)|^p\right)^{\frac{1}{p}}
\end{equation}
and
\begin{eqnarray}\label{i}
\left(\mathbb{E}_\xi\left|V_1(z,z')\right|^p\right)^{\frac{1}{p}}
&=&\left(\mathbb E_\xi\left|\sum_{u\in \mathbb T_1}\tilde X_1(z+z')\left(W(u,z)-\mathbb{E}_\xi W(u,z)+\mathbb{E}_\xi W(u,z)\right)\right|^p\right)^{1/p}\nonumber\\
&\leq&\left(\mathbb E_\xi\left|\sum_{u\in \mathbb T_1}\tilde X_1(z+z')\left(W(u,z)-\mathbb{E}_\xi W(u,z)\right)\right|^p\right)^{1/p}\nonumber\\
&&+\left(\mathbb E_\xi\left|W_1(z+z')\right|^p\right)^{1/p}|\mathbb{E}_{T\xi} W(z)|.
\end{eqnarray}
Burkholder's inequality yields
\begin{eqnarray}\label{LDMDS3p3e7}
\mathbb E_\xi\left|\sum_{u\in \mathbb T_1}\tilde X_1(z+z')\left(W(u,z)-\mathbb{E}_\xi W(u,z)\right)\right|^p
%&\leq&C\mathbb E_\xi\sum_{u\in \mathbb T_1}|\tilde X_1(z+z')|^p\left|W(u,z)-\mathbb{E}_\xi W(u,z)\right|^p\nonumber\\
\leq C\frac{m_0(p(x+x'))}{|m_0(z+z')|^p}\mathbb{E}_{T\xi}| W(z)|^p.
\end{eqnarray}
Combing (\ref{LDMDS3p3e4})-(\ref{LDMDS3p3e7}), we obtain
\begin{eqnarray*}\label{LDMDS3p3e8}
&&\sup_{\substack {z\in D(t_0,\varepsilon)\\z'\in D(0,\varepsilon)}}\left(\mathbb E_\xi|V_{n+1}(z,z')-V_n(z,z')|^p\right)^{1/p}\nonumber\\
&\leq&C\left(\sup_{z\in D(t_0,2\varepsilon)}\frac{P_{n+1}(px)^{1/p}}{|P_{n+1}(z)|}\sup_{z\in D(t_0,2\varepsilon)}(\mathbb E_{T^{n+1}\xi}|W(z)|^p)^{1/p}\right.\nonumber\\
&&+\sup_{z\in D(t_0,2\varepsilon)}\frac{P_n(px)^{1/p}}{|P_n(z)|}\sup_{z\in D(t_0,2\varepsilon)}(\mathbb E_{T^n\xi}|W_1(z)|^p)^{1/p}
\sup_{z\in D(t_0,2\varepsilon)}\mathbb E_{T^{n+1}\xi}|W(z)|\nonumber\\
&&+\left.\sup_{z\in D(t_0,2\varepsilon)}\frac{P_n(px)^{1/p}}{|P_n(z)|}\sup_{z\in D(t_0,2\varepsilon)}(\mathbb E_{T^n\xi}|W(z)|^p)^{1/p}\right).
\end{eqnarray*}
Thus
\begin{eqnarray}\label{LDMDS3p3e9}
&&\sum_n\sup_{\substack {z\in D(t_0,\varepsilon)\\z'\in D(0,\varepsilon)}}\left(\mathbb E_\xi|V_{n+1}(z,z')-V_n(z,z')|^p\right)^{1/p}\nonumber\\
&\leq&C\left(\sum_n \beta_0\cdots\beta_{n-1}\sup_{z\in D(t_0,2\varepsilon)}(\mathbb E_{T^n\xi}|W(z)|^p)^{1/p}\right.+\left.\sum_n \beta_0\cdots\beta_{n-1}\tilde\gamma_n
\sup_{z\in D(t_0,2\varepsilon)}\mathbb E_{T^{n+1}\xi}|W(z)|\right),\nonumber\\
\end{eqnarray}
where $\beta_n$ is defined in (\ref{LDMAS3pt1ee}) with $\varepsilon$ replaced by $2\varepsilon$, and $\tilde\gamma_n=\underset{z\in D(t_0,2\varepsilon)}\sup(\mathbb E_{T^n\xi}|W_1(z)|^p)^{1/p}$. Lemmas \ref{LDMAl4} and \ref{LDMAl5} ensure that  $\mathbb E\log\beta_0<0$ and $\mathbb E\log^+\tilde\gamma_0<\infty$ for suitable  $0<\varepsilon\leq\delta$ and $p\in(1,2]$. Therefore, the first series in  right hand side of the inequality (\ref{LDMDS3p3e9}) above converges a.s. by Lemma \ref{LDMAl7}. For the second series, notice that
$$\limsup_{n\rightarrow\infty}\frac{1}{n}\log \left(\beta_0\cdots\beta_{n-1}\tilde \gamma_n\right)\leq\mathbb E\log\beta_0<0,$$
which implies that $\beta_0\cdots\beta_{n-1}\tilde \gamma_n<a^{-n}$ a.s. for some constant $a>1$ as $n$ large enough.
It follows from Lemma \ref{LDMAl7} that the series $\sum_na^{-n}\underset{z\in D(t_0,2\varepsilon)}\sup\mathbb E_{T^{n+1}\xi}|W(z)|<\infty$ a.s., which implies the a.s. convergence of the second series in the right hand side of (\ref{LDMDS3p3e9}).
\end{proof}

Following the proof of (\cite{na}, Corollary 2.5), Proposition \ref{LDMDS3p3} leads to the proposition below.

\begin{pr}\label{LDMDS3p1}
With probability $1$, for all $t\in\Omega$,
$$\lim_{n\rightarrow\infty}\frac{S_{u|n}}{n}=\nabla\Lambda(t)\quad\text{
for $\mu_t$-a.s. $u\in\partial \mathbb T$.}$$
\end{pr}

\medskip
\begin{proof}[Proof of Theorem \ref{LS1.4TT1}]
Letting $n$ tends to infinity in (\ref{LDMDS4em}) and using Propositions \ref{LDMDS3p2} and \ref{LDMDS3p1}, we immediately obtain the desired result.
\end{proof}

\medskip
We will use Theorem  \ref{LS1.4TT1} to calculate the lower bounds of the Hausdorff dimensions $\dim E(\alpha)$,  so as to further achieve the proof of  Theorem \ref{Conver1.12ma}.

\begin{proof}[Proof of Theorem \ref{Conver1.12ma}]
 By Propositions \ref{LDMDS3l1} and \ref{LDMDSSpp2},  with probability $1$, we have $\dim E(\alpha)\leq -\tilde \Lambda^*(\alpha)\leq - \Lambda^*(\alpha)$ for all $\alpha\in\mathbb R^d$. On the other hand, Proposition \ref{LDMDS3p1} shows that with probability $1$, for all $\alpha\in\mathcal J$ (so that $\alpha=\nabla\Lambda(t)$ for some $t\in\Omega$), $0<\mu_t(u\in\partial\mathbb T:\lim_{n}\frac{S_{u|n}}{n}=\alpha)$ ($<\infty$).
Noticing Theorem \ref{LS1.4TT1} and using (\cite{fal}, Theorem 4.2), we deduce $\dim E(\alpha)\geq \Lambda(t)-\langle t,\nabla\Lambda(t)\rangle = - \Lambda^*(\alpha)$.
\end{proof}

\section{Proof of Corollary \ref{Conver1.12cma}}\label{LDMDSS3}
In order to give the proof of Corollary \ref{Conver1.12cma}, we need a technique of truncation. Letting $a$ be a positive rational number, we  introduce the point process related to $u\in\mathbb T$ as $X^a(u)=(N^a(u), L^a_{1}(u), L^a_{2}(u), \cdots)$, where $N^a(u)=N(u)\wedge a$ with notation $a_1\wedge a_2=\min(a_1,a_2)$, $L^a_{i}(u)$ equals to $L_i(u)$ if $\|L_i(u)\|\leq a$ and is empty otherwise. Let us construct a new BRWRE
where the point process formed by a particle $u$ is $X^a(u)$. Denote
$$m_0^a(z)=\mathbb{E}_\xi\sum^{N\wedge a}_{i=1}e^{\langle z,{L_i^a}\rangle}=\mathbb{E}_\xi\sum^{N\wedge a}_{i=1}e^{\langle z,{L_i}\rangle}\mathbf{1}_{\{\|L_i\|\leq a\}}
\qquad (z\in\mathbb C^d),$$
%$$\Lambda_c(t)=\mathbb E\log m^c_0(t)\qquad (t\in\mathbb R^d),$$
and the other notations can be extended similarly. If $\mathbb E \left[\mathbb P_\xi (\|L_1\|\leq a_0)^{-1}\right]<\infty$ for some constant $a_0>0$, it is not difficult to verify that for all $t\in\mathbb R^d$ and $a\geq a_0$,
$$\mathbb E|\log m_0^a(t)|<\infty\quad\text{and}\quad \mathbb E\left\|\frac{\nabla m_0^a(t)}{m_0^a(t)}\right\|<\infty ,$$
which ensures that the function $\Lambda_a(t)=\mathbb E\log m^a_0(t)$ is well defined as real number on $\mathbb R^d$ and  differential everywhere. Indeed, notice that $\log^+ m_0^a(t)\leq \log^+ m_0(t)$ and for $a\geq a_0$,
\begin{eqnarray}\label{LDMAPC1}
m_0^a(t)\geq \mathbb E_\xi e^{\langle t, L_1\rangle}\mathbf{1}_{\{\|L_1\|\leq a\}}\geq \mathbb E_\xi e^{\langle t, L_1\rangle}\mathbf{1}_{\{\|L_1\|\leq a_0\}}\geq  e^{-\|t\|a_0}\mathbb P_\xi(\|L_1\|\leq a_0),
\end{eqnarray}
so that
\begin{eqnarray*}\label{LDMAPCE2}
\label{LDMAPCE1}\log^- m_0^a(t)\leq \|t\|a_0-\log \mathbb P_\xi (\|L_1\|\leq a_0)=\|t\|a_0+\log\left[ \mathbb P_\xi (\|L_1\|\leq a_0)^{-1}\right].
\end{eqnarray*}
The fact that $\mathbb E\log^+ m_0(t)<\infty$ and (by Jensen's inequality)
\begin{equation}\label{LDMAPCE3}
\mathbb E\log \left[ \mathbb P_\xi (\|L_1\|\leq a_0)^{-1}\right]\leq \log \mathbb E\left[ \mathbb P_\xi (\|L_1\|\leq a_0)^{-1}\right]<\infty
\end{equation}
ensures $\mathbb E|\log m_0^a(t)|<\infty$, and also implies that $\Lambda_a(t)\uparrow \Lambda(t)$ as $a\uparrow\infty$. Besides, by (\ref{LDMAPC1}), we can  deduce
$$\left\|\frac{\nabla m_0^a(t)}{m_0^a(t)}\right\|\leq\frac{\mathbb E_\xi \sum_{i=1}^{N\wedge a}e^{\langle t, L_i\rangle}\|L_i\|\mathbf{1}_{\{\|L_i\|\leq a\}}}{e^{-\|t\|a_0}\mathbb P_\xi(\|L_1\|\leq a_0)}\leq a^2 e^{(a-a_0)\|t\|}\mathbb P_\xi (\|L_1\|\leq a_0)^{-1},$$
so $\mathbb E\left\|\frac{\nabla m_0^a(t)}{m_0^a(t)}\right\|<\infty$ since $\mathbb E[\mathbb P_\xi (\|L_1\|\leq a_0)^{-1}]<\infty$.

%\begin{lem}\label{LDMDS4L1}
%Assume that $\mathbb E \left[\mathbb P_\xi (\|L_1\|\leq c_0)^{-1}\right]<\infty$ for some constant $c_0>0$. For $c\geq c_0$, if $\alpha\in int\{\alpha\in\mathbb R^d: \Lambda^*_c(\alpha)<0\}$, then there exists $t\in I_c=\{t\in\mathbb R^d: \langle t,\nabla\Lambda_c(t)\rangle-\Lambda_c(t)<0\}$ such that $\alpha=\nabla\Lambda_c(t)$.
%\end{lem}

From the above, we see that $\Lambda_a(t)\uparrow \Lambda(t)$ as $a\uparrow\infty$ as soon as $\mathbb E\log\mathbb P_\xi(\|L_i\|\leq a_0)>-\infty$ for some constant $a_0>0$. According to the arguments in (\cite{na2014}, proofs of Proposition 2.3 and Corollary 2.1), we can get the following lemma, which is a generalization of (\cite{na2014}, Corollary 2.1)  for $\mathbb R^d$-valued BRWRE.
\begin{lem}\label{LDMDS4L2}
Assume that $\mathbb E\log\mathbb P_\xi(\|L_i\|\leq a_0)>-\infty$ for some constant $a_0>0$. If $\alpha\in \emph{int}\{\alpha: \Lambda^*(\alpha)<\infty\}$, then $\Lambda^*_a(\alpha)\downarrow \Lambda^*(\alpha)$ as $a\uparrow\infty$.%\marginpar{so by Dini's theorem, uniform convergence on closed set }
\end{lem}

\begin{lem}\label{LDMDS4L3}Assume that $\mathbb E\log\mathbb P_\xi(\|L_i\|\leq a_0)>-\infty$ for some constant $a_0>0$. Then $I_a=\Omega_a$ for $a\geq a_0$,
where $\Omega_a=I_a\bigcap\Omega_1^a\bigcap \Omega_2^a$.
\end{lem}

\begin{proof}
We shall prove that $\Omega_1^a=\Omega_2^a=\mathbb R^d$ for $a\geq a_0$, %where we write $\tilde\Omega_1^c=\{t\in\mathbb{R}^d:\mathbb{E}\log^+\mathbb{E}_\xi\tilde{Z}^c_1(t)^p<\infty~for~some~p>1\}$,
%then we get $\Omega_1^c=\Omega_2^c=\mathbb R^d$,
which implies $I_a=\Omega_a$.
On the one hand, for $t\in \mathbb R^d$, one can see that
$$\mathbb E_\xi [\tilde{Z}^a_1(t)]^p=\mathbb E_\xi\left(\sum_{i=1}^{N\wedge a}e^{\langle t, L_i\rangle}\mathbf{1}_{\{\|L_i\|\leq a\}}\right)^p\leq a^pe^{pa\|t\|},$$
which implies that $\mathbb{E}\log^+\mathbb{E}_\xi[\tilde{Z}^a_1(t)]^p<\infty$, hence we have $\Omega_1^a=\mathbb R^d$ (since $\mathbb R^d$ is open).

On the other hand,  notice that the function $h(\theta)=\sin\theta+\cos\theta\geq\frac{1}{2}$ on $[-\delta_1,\delta_1]$ for some $\delta_1>0$. Take $\delta>0$ small enough such that  $a \sqrt{d}\delta\leq \delta_1$. For $t\in\mathbb R^d$ and $z=x+\mathbf{i}y\in D(t,\delta)$, we have $\|x\|\leq \|x-t\|+\|t\|<\sqrt{d}\delta+\|t\|$ and $\|y\|<\sqrt{d}\delta$. Therefore,  for $a\geq a_0$,
\begin{eqnarray*}
|m_0^a(z)|^2&=&\left(\mathbb E_\xi\sum^{N\wedge a}_{i=1}e^{\langle x,{L_i}\rangle}\cos\langle y,{L_i}\rangle\mathbf{1}_{\{\|L_i\|\leq a\}} \right)^2
+\left(\mathbb E_\xi\sum^{N\wedge a}_{i=1}e^{\langle x,{L_i}\rangle}\sin\langle y,{L_i}\rangle\mathbf{1}_{\{\|L_i\|\leq a\}} \right)^2\\
&\geq& \frac{1}{2}\left(\mathbb E_\xi\sum^{N\wedge a}_{i=1}e^{\langle x,{L_i}\rangle}h(\langle y,{L_i}\rangle)\mathbf{1}_{\{\|L_i\|\leq a\}} \right)^2\\
%&\geq&  \frac{1}{8}\left(\mathbb E_\xi\sum^{N\wedge a}_{i=1}e^{\langle x,{L_i}\rangle}\mathbf{1}_{\{\|L_i\|\leq a\}} \right)^2\\
%&\geq&  \frac{1}{8}\left(\mathbb E_\xi\sum^{N\wedge a_0}_{i=1}e^{-\|x\|a_0} \mathbf{1}_{\{\|L_i\|\leq a_0\}} \right)^2\\
&\geq& \frac{1}{8}e^{-2a_0(\sqrt{d}\delta+\|t\|)} \mathbb P_\xi(\|L_i\|\leq a_0)^2,
\end{eqnarray*}
so that $$\mathbb E\log^-\alpha_0^a(t,\delta)\leq 2\sqrt{2}+{a_0(\sqrt{d}\delta+\|t\|)} -\mathbb E\log\mathbb P_\xi(\|L_i\|\leq a_0)<\infty.$$
Thus $ \Omega_2^a=\mathbb R^d$. The proof is finished.
%\mathbb E\log^-\underset{\zeta\in D(t,\delta)}\inf|m^c_0(z)|
\end{proof}

\begin{proof}[Proof of Corollary \ref{Conver1.12cma}]We only need to show that with probability $1$,  for all $\alpha\in \text{int}\tilde {\mathcal J}$, the lower bound of the Hausdorff dimension of the set $E(\alpha)$ satisfies $\dim E(\alpha)\geq -\Lambda^*(\alpha)$. Firstly, it is obvious that the tree $\mathbb T^a\subset \mathbb T$, thus
$$E_a(\alpha)=\{u\in\partial\mathbb T^a:\lim_{n\rightarrow\infty}\frac{S_{u|n}}{n}=\alpha\}  \subset E(\alpha),$$
which leads to $\dim E(\alpha)\geq \dim E_a(\alpha)$ for all $\alpha\in\mathbb R^d$.
Denote
$$\widetilde{\mathcal E}=\bigcap_a\{\omega\in\Theta^{\mathbb{N}}\times\Gamma: \dim E_a(\alpha)=\Lambda_a^*(\alpha),\quad \forall \alpha\in \mathcal J_a\}.$$
Then $\mathbb P(\widetilde{\mathcal E})=1$ by Theorem \ref{Conver1.12ma}. Fix $\omega\in \widetilde{\mathcal E}$. For $\alpha\in \text{int}\tilde {\mathcal J}$$\subset \text{int}\{\alpha:\Lambda^*(\alpha)<\infty\}$, by Lemma \ref{LDMDS4L2}, we have $\Lambda_a^*(\alpha)<0$ for $a$ large enough. Take  $a$ large enough.  Since $\Lambda_a(t)$ is differentiable,  there exists $t_\alpha\in I_a$ such that $\alpha=\nabla\Lambda_a(t_\alpha)$ (see \cite{r70}, p227).  Lemma \ref{LDMDS4L3} shows $I_a=\Omega_a$, so  that $\alpha\in \mathcal J_a$. Thus $\dim E(\alpha)\geq\dim E_a(\alpha)= -\Lambda_a^*(\alpha)$ for $a$ large enough. Letting $a\uparrow\infty$ gives $\dim E(\alpha)\geq -\Lambda^*(\alpha)$. Hence $\widetilde{\mathcal E}\subset\{\omega:\dim E(\alpha)\geq -\Lambda^*(\alpha),\;\;\forall\alpha\in \text{int}\tilde{\mathcal J}\}$.
\end{proof}
%%%%%%%%%%%%%%%%%%%%%%%%%   section 4   %%%%%%%%%%%%%%%%%%%%%%%%%%%%%%%%%%%%%%
\section{Proof of Theorem \ref{Mdp1.4}}\label{LDMDS4}

We will prove  Theorem \ref{Mdp1.4} along the lines of the proof of (\cite{huang17}, Theorem 5.1) with details modified.
For $n\in\mathbb N$, set $\ell_n=\frac{1}{\pi_n}\mathbb E_\xi\sum\limits_{j=1}^{N(u)}L_j(u)$ with $u\in\mathbb T_n$ and $\mu_n=\sum\limits_{i=0}^{n-1}\ell_i$.
Consider the probability measures $q_n(\cdot)=\frac{\mathbb{E}_\xi Z_n(a_n\cdot+\mu_n)}{\mathbb{E}_\xi Z_n(\mathbb{R}^d)}$. For $t\in\mathbb R^d$, put
\begin{eqnarray}\label{16}
\lambda_n(t)=\log\int e^{\langle t,x\rangle}q_n(dx)
%=\log\left[\frac{\mathbb{E}_\xi \int e^{\langle t,\frac{y}{a_n}\rangle}Z_n(dy)}{\mathbb{E}_\xi Z_n(\mathbb{R}^d)}\right]
%=\log\frac{P_n(a_n^{-1}t)}{P_n(0)}
=\sum\limits_{i=0}^{n-1}\log\left[\frac{1}{\pi_i}m_i(a_n^{-1}t )e^{-\langle a_n^{-1}t,\;\ell_i\rangle}\right].
\end{eqnarray}

\begin{lem}\label{LDMDMDL1}If $\limsup_n\frac{a_n}{n^\beta}<\infty$ for some $\beta\in(\frac{1}{2},1)$ and $\mathbb E\frac{1}{\pi_0} \underset{u\in \mathbb T_1}{\Sigma}e^{\delta\|S_u\|}<\infty$  for some $\delta>0$,  then for each $t\in\mathbb{R}^d$,
\begin{equation}\label{17}
\underset{n\rightarrow\infty}{\lim}\frac{n}{a^2_n}\lambda_n\left(\frac{a^2_n}{n}t\right)=\Gamma(t)~~~~~\text{a.s.,}
\end{equation}
 where the   function $\Gamma(t)$ is defined in  Theorem \ref{Mdp1.4}.
\end{lem}

\begin{proof}
Put $\Delta_{n,i}=\frac{1}{\pi_i}m_i(\frac{a_n}{n}t)e^{-\langle  \frac{a_n}{n}t,\;\ell_i\rangle}-1$. We shall show that  for each $t\in\mathbb{R}^d$,
\begin{equation}\label{18}
\underset{0\leq i\leq n-1}{\sup}|\Delta_{n,i}|<1~~~~~\text{a.s.}
\end{equation}
for $n$   large enough.  Set $\bar L_i(u)=L_i(u)-{\ell_{|u|}}$
and denote
 %\begin{equation}
$\bar X_n(\cdot) = \sum\limits_{i=1}^{N(u)}   \delta_{\bar L_i (u)} (\cdot)$  ($u\in\mathbb T_n$)
%\end{equation}
the counting measure  corresponding to the random vector $X(u)$.
Let
$Q_n^{(\epsilon)}%=\frac{1}{\pi_n}\mathbb E_\xi \int e^{\delta\|x\|}X_n(dx)
=\frac{1}{\pi_n}\mathbb E_\xi\sum\limits_{i=1}^{N(u)}e^{\epsilon\|L_i(u)\|}$ and $\bar Q_n^{(\epsilon)}%=\frac{1}{\pi_n}\mathbb E_\xi \int e^{\delta\|x\|}X_n(dx)
=\frac{1}{\pi_n}\mathbb E_\xi\sum\limits_{i=1}^{N(u)}e^{\epsilon\|\bar L_i(u)\|}$  $(u\in\mathbb T_n)$. By the triangle inequality and Jensen's inequality,
we see that
$$\bar Q_n^{(\epsilon)}\leq Q_n^{(\epsilon)} \exp\{\frac{2^d}{\pi_n}\mathbb E_\xi\sum\limits_{i=1}^{N(u)}\epsilon\|L_i(u)\| \}\leq Q_n^{(\epsilon)}\frac{1}{\pi_n}\mathbb E_\xi\sum\limits_{i=1}^{N(u)}
e^{2^d\epsilon\|L_i(u)\|}\leq( Q_n^{(\delta)})^2,$$
for $\epsilon=\delta/2^d$.
By the ergodic theorem,
$\lim_{n}\frac{1}{n}\sum\limits_{i=0}^{n-1}Q_i^{(\delta)}=\mathbb E Q_0^{(\delta)}<\infty$ a.s., hence for $n$ large enough,
\begin{equation}\label{LDMDS5p1}
\sup_{0\leq i\leq n-1}\bar Q_i^{(\epsilon)}\leq\sup_{0\leq i\leq n-1}(Q_i^{(\delta)})^2\leq \left(\sum\limits_{i=0}^{n-1}Q_i^{(\delta)}\right)^2\leq Cn^2\qquad a.s..
\end{equation}
Notice that for $n$ large  enough (such that $\frac{a_n}{n}\|t\|<\epsilon$),
\begin{eqnarray*}
\overset{\infty}{\underset{k=0}{\sum}}\frac{1}{\pi_i}\mathbb{E}_\xi\int\frac{1}{k!}\left|\langle\frac{a_n}{n}t,x\rangle\right|^k \bar X_i(dx)
%&=&\frac{1}{\pi_i}\mathbb{E}_\xi\int\overset{\infty}{\underset{k=0}{\sum}}\frac{1}{k!}\left|\langle\frac{a_n}{n}t,x\rangle\right|^kX_i(dx)\nonumber\\
%=\frac{1}{\pi_i}\mathbb{E}_\xi\int e^{|\langle\frac{a_n}{n}t,x\rangle|}X_i(dx)
%&\leq&\frac{1}{\pi_i}\mathbb{E}_\xi\int e^{\frac{a_n}{n}\|t\|\|x\|}X_i(dx)\nonumber\\
\leq\frac{1}{\pi_i}\mathbb{E}_\xi\int e^{\epsilon\|x\|}\bar X_i(dx)=\bar Q_i^{(\epsilon)}<\infty~~\text{a.s..}
\end{eqnarray*}
Since $\mathbb{E}_\xi\sum\limits_{j=1}^{N(u)}\bar L_j(u)=\mathbb{E}_\xi\sum\limits_{j=1}^{N(u)} L_j(u)-\ell_{|u|}\pi_{|u|} =0$ a.s.,  we can write $\Delta_{n,i}$ as
\begin{eqnarray*}
\Delta_{n,i}
%&=&\frac{1}{\pi_i}m_i\left(\frac{a_n}{n}t\right)-1\nonumber\\
%&=&\frac{1}{\pi_i}\mathbb{E}_\xi\int e^{\langle\frac{a_n}{n}t,x\rangle}X_i(dx)-1\nonumber\\
=\frac{1}{\pi_i}\mathbb{E}_\xi\int\overset{\infty}{\underset{k=0}{\sum}}\frac{\langle\frac{a_n}{n}t,x\rangle^k}{k!}\bar X_i(dx)-1%\nonumber\\
=\overset{\infty}{\underset{k=0}{\sum}}\frac{1}{k!}\frac{1}{\pi_i}\mathbb{E}_\xi\int\langle\frac{a_n}{n}t,x\rangle^k\bar X_i(dx)-1%\nonumber\\
%&=&\overset{\infty}{\underset{k=2}{\sum}}\frac{1}{k!}\frac{1}{\pi_i}\mathbb{E}_\xi\int\langle\frac{a_n}{n}t,x\rangle^kX_i(dx)
=\overset{\infty}{\underset{k=2}{\sum}}\gamma^n_{ik},
\end{eqnarray*}
with the notation
\begin{eqnarray*}
\gamma^n_{ik}
&=&\frac{1}{k!}\frac{1}{\pi_i}\mathbb{E}_\xi\int\langle\frac{a_n}{n}t,x\rangle^k\bar X_i(dx)\nonumber\\
&=&\frac{1}{k!}\frac{1}{\pi_i}\mathbb{E}_\xi\int\langle\frac{a_n}{n}t,x\rangle^k\mathbf{1}_{\{\|x\|\leq \frac{4}{\epsilon}\log n\}}\bar X_i(dx)+\frac{1}{k!}\frac{1}{\pi_i}\mathbb{E}_\xi\int\langle\frac{a_n}{n}t,x\rangle^k\mathbf{1}_{\{\parallel x\parallel>\frac{4}{\epsilon}\log n\}}\bar X_i(dx)\nonumber\\
&=&:\alpha^n_{ik}+\beta^n_{ik}.
\end{eqnarray*}
 We can calculate that
\begin{eqnarray}\label{LDMDS5p2}
|\alpha^n_{ik}|
%&=&\frac{1}{k!}\frac{1}{\pi_i}\mathbb{E}_\xi\int\left|\langle\frac{a_n}{n}t,x\rangle^k\right|I_{\{\|x\|\leq C_n\}}X_i(dx)\nonumber\\
\leq\frac{1}{k!}\left(\frac{a_n}{n}\|t\|\right)^k\frac{1}{\pi_i}\mathbb{E}_\xi\int\|x\|^k\mathbf{1}_{\{\|x\|\leq \frac{4}{\epsilon}\log n\}}\bar X_i(dx)%\nonumber\\
%&\leq&\frac{1}{k!}\frac{1}{\pi_i}\mathbb{E}_\xi\int\left(\frac{a_n}{n}\|t\|\right)^k C^k_nX_i(dx)\nonumber\\
%&=&\frac{1}{k!}\frac{1}{\pi_i}\left(\frac{a_n}{n}C_n\|t\|\right)^k\mathbb{E}_\xi\int X_i(dx)\nonumber\\
\leq\frac{1}{k!}\left(\frac{a_n}{n}\frac{4}{\epsilon}\log n\|t\|\right)^k
\end{eqnarray}
and
\begin{eqnarray}\label{LDMDS5p3}
|\beta^n_{ik}|
\leq\frac{1}{k!}\left(\frac{a_n}{n}\|t\|\right)^k\frac{1}{\pi_i}\mathbb{E}_\xi\int\|x\|^k n^{-2}e^{\frac{\epsilon}{2}\|x\|}\bar X_i(dx)%\nonumber\\
%&\leq&\frac{1}{k!}\left(\frac{a_n}{n}\|t\|\right)^k\frac{1}{\pi_i}\mathbb{E}_\xi\int\|x\|^k  \frac{e^{\frac{\delta}{2}\|x\|}}{e^{\frac{\delta}{2}c_n}}X_i(dx)\nonumber\\
\leq\left(\frac{2}{\epsilon}\frac{a_n}{n}\|t\|\right)^k\frac{\bar Q_i^{(\epsilon)}}{n^2}.%\frac{1}{\pi_i}\mathbb{E}_\xi\int e^{\delta\|x\|}X_i(dx)
\end{eqnarray}
In the last %line of the
inequality above, we have used the fact that $\frac{1}{k!}\left(\frac{\epsilon}{2}\|x\|\right)^k\leq e^{\frac{\epsilon}{2}\|x\|}$ for all $k$. By (\ref{LDMDS5p2}), (\ref{LDMDS5p3}) and (\ref{LDMDS5p1}), we see that for $n$ large enough,

%$$\sup_{1\leq i\leq n-1}|\gamma^n_{ik}|\leq \sup_{1\leq i\leq n-1}|\alpha^n_{ik}|+\sup_{1\leq i\leq n-1}|\beta^n_{ik}|\leq
%\left(\frac{2}{\delta}\frac{a_n}{n}\log n\|t\|\right)^k+C\left(\frac{2}{\delta}\frac{a_n}{n}\|t\|\right)^k\leq
%C\left(\frac{2}{\delta}d_n\|t\|\right)^k\qquad a.s.,$$
%Therefore, for   $n$ large enough,
\begin{equation}\label{LDMDS5p41}
\sup_{1\leq i\leq n-1}|\Delta_{n,i}|\leq\sum_{k=2}^\infty\sup_{1\leq i\leq n-1}|\gamma^n_{ik}|\leq C\sum_{k=2}^\infty
\left(\frac{4}{\epsilon} d_n\|t\|\right)^k\leq M_1 d_n^2\qquad a.s.,
\end{equation}
where $d_n=\frac{a_n}{n}\log n$ and $M_1>0$ is a constant (depending on $t$). It is clear that $\lim_n d_n=0$, so that (\ref{18}) holds for $n$ sufficiently large.

Now we calculate (\ref{17}).   Noticing
(\ref{18}),  when $n$ is  large enough,  we can write a.s.,
\begin{eqnarray*}
\frac{n}{a^2_n}\lambda_n\left(\frac{a^2_n}{n}t\right)
%&=&\frac{n}{a^2_n}\overset{n-1}{\underset{i=0}{\sum}}\log(1+\Delta_{n,i})\nonumber\\
%&=&\frac{n}{a^2_n}\overset{n-1}{\underset{i=0}{\sum}}\overset{\infty}{\underset{j=1}{\sum}}\frac{(-1)^{j+1}}{j!}(\Delta_{n,i})^j\nonumber\\
%&=&\frac{n}{a^2_n}\overset{\infty}{\underset{j=1}{\sum}}\frac{(-1)^{j+1}}{j!}\overset{n-1}{\underset{i=0}{\sum}}(\Delta_{n,i})^j\nonumber\\
&=&\frac{n}{a^2_n}\overset{n-1}{\underset{i=0}{\sum}}\Delta_{n,i}+\frac{n}{a^2_n}\overset{\infty}{\underset{j=2}{\sum}}\frac{(-1)^{j+1}}{j!}\overset{n-1}{\underset{i=0}{\sum}}(\Delta_{n,i})^j%\nonumber\\
=:A_n+B_n.
\end{eqnarray*}
For $B_n$, by (\ref{LDMDS5p41}), we get for $n$ large enough,
\begin{equation}\label{LDMDS5p5}
|B_n|\leq\frac{n}{a^2_n}\overset{\infty}{\underset{j=2}{\sum}}\frac{1}{j}\overset{n-1}{\underset{i=0}{\sum}}|\Delta_{n,i}|^j
\leq\frac{n^2}{a^2_n}\overset{\infty}{\underset{j=2}{\sum}}\frac{1}{j}M_1^j d^{2j}_n\leq\frac{n^2}{a^2_n}\overset{\infty}{\underset{j=2}{\sum}}M_1^j d^{2j}_n\leq M_2 \frac{n^2}{a^2_n}d_n^4\qquad a.s.,
\end{equation}
where $M_2>0$ is a constant.
Since $\lim_n d_n=0$ and
$
\lim_{n} \frac{n^2}{a^2_n}d_n^3=
\lim\limits_{n} \frac{a_n}{n^\beta}\frac{(\log n)^3}{n^{1-\beta}}=0
$,
we deduce $B_n\rightarrow0$ a.s. immediately  from (\ref{LDMDS5p5}). For $A_n$, %noticing that the series
%$\overset{\infty}{\underset{k=2}{\sum}}\overset{n-1}{\underset{i=0}{\sum}}|\gamma^n_{ik}|<\infty$ a.s. for $n$ large enough,
we can decompose
\begin{eqnarray*}
A_n
%&=&\frac{n}{a^2_n}\overset{n-1}{\underset{i=0}{\sum}}\Delta_{n,i}
=\frac{n}{a^2_n}\overset{n-1}{\underset{i=0}{\sum}}\overset{\infty}{\underset{k=2}{\sum}}\gamma^n_{ik}
=\frac{n}{a^2_n}\overset{\infty}{\underset{k=2}{\sum}}\overset{n-1}{\underset{i=0}{\sum}}\gamma^n_{ik}%\nonumber\\
=\frac{n}{a^2_n}\overset{n-1}{\underset{i=0}{\sum}}\gamma^n_{i2}+\frac{n}{a^2_n}\overset{\infty}{\underset{k=3}{\sum}}\overset{n-1}{\underset{i=0}{\sum}}\gamma^n_{ik}
=: C_n+D_n.
\end{eqnarray*}
For $D_n$, %by (\ref{LDMDS5p4}),
 for $n$ large enough,
$$|D_n|\leq\frac{n}{a^2_n}\overset{\infty}{\underset{k=3}{\sum}}\overset{n-1}{\underset{i=0}{\sum}}|\gamma^n_{ik}|\leq C\frac{n^2}{a^2_n} \overset{\infty}{\underset{k=3}{\sum}}\left(\frac{4}{\epsilon} d_n\|t\|\right)^k\leq M_3\frac{n^2}{a^2_n} d_n^3\qquad a.s.,$$
where $M_3>0$ is a constant, so that $D_n\rightarrow0$ a.s.. Finally, it remains to calculate the limit of $C_n$. By the ergodic theorem,
\begin{eqnarray*}
\lim_{n\rightarrow\infty}C_n
%&=&\frac{n}{a^2_n}\overset{n-1}{\underset{i=0}{\sum}}r_{i2}
=\lim_{n\rightarrow\infty}\frac{n}{a^2_n}\overset{n-1}{\underset{i=0}{\sum}}\frac{1}{2}\frac{1}{\pi_i}\mathbb{E}_\xi\int\langle\frac{a_n}{n}t,x\rangle^2\bar X_i(dx)
%=\frac{1}{2}\lim_{n\rightarrow\infty}\frac{1}{n}\overset{n-1}{\underset{i=0}{\sum}}\frac{1}{\pi_i}\mathbb{E}_\xi\int\langle t,x\rangle^2X_i(dx)
%=\frac{1}{2}\mathbb{E}\frac{1}{\pi_0}\int\langle t,x\rangle^2X_0(dx)
=\Gamma(t)\;\; a.s.,
\end{eqnarray*}
which completes the proof.
\end{proof}

\medskip
Applying Lemma \ref{LDMDMDL1} and the uniform convergence of $W_n(t)$ near $0$, we carry on the proof of Theorem \ref{Mdp1.4}.
\begin{proof}[Proof of Theorem \ref{Mdp1.4}]
Let
$\Gamma_n(t)=\log\left[\frac{\int e^{\langle\frac{t}{a_n},x\rangle}Z_n(dx)}{Z_n(\mathbb{R}^d)}\right]=\log\left[\frac{\tilde Z_n(a_n^{-1}t)}{Z_n(\mathbb R^d)}\right]$.
It can be seen that
\begin{equation}\label{MEP1}
\frac{n}{a_n^2}\Gamma_n(\frac{a_n^2}{n}t)=\frac{n}{a_n^2}\log W_n(\frac{a_n}{n}t)+\frac{n}{a_n^2}\lambda_n(\frac{a_n^2}{n}t)+\frac{1}{a_n}\sum_{i=0}^{n-1}\langle t,\ell_i\rangle
-\frac{n}{a_n^2}\log W_n(0).
\end{equation}
It is evident that $0\in I$, since $-\Lambda(0)=-\mathbb E\log m_0(0)<0$. So we have $0\in\Omega$. By  Theorem \ref{Conver1.1}, $W_n(z)$ converges uniformly a.s. in a neighbourhood of $0\in\mathbb C^d$, so that the limit $W(z)$ is  continuous at $0$.
Since the environment $\xi$ is a stationary mixing sequence satisfying $\mathbb E \ell_0=0$, by (\cite{hipp79}, Theorem 2), we have $\lim_{n}\frac{1}{a_n}\sum\limits_{i=0}^{n-1}\langle t,\ell_i\rangle=0$ a.s..
%Therefore,
%$$\lim_{n\rightarrow\infty}W_n(\frac{a_n}{n}t)= W(0)>0\qquad a.s..$$
Letting $n\rightarrow\infty$ in (\ref{MEP1}) and using Lemma \ref{LDMDMDL1},
we obtain for each $t\in\mathbb R^d$,
\begin{equation}\label{MEP2}
\lim_{n\rightarrow}\frac{n}{a_n^2}\Gamma_n(\frac{a_n^2}{n}t)
=\Gamma(t)\qquad a.s..
\end{equation}
So (\ref{MEP2}) a.s. holds for all rational $t$, and hence for all $t\in\mathbb R^d$ by the convexity of $\Gamma_n(t)$ and the continuity of $\Gamma(t)$.
Then apply the G\"{a}rtner-Ellis theorem.
\end{proof}

%\noindent\bf{ Acknowledgements}\quad\rm
%{The authors would like to thank the anonymous referees for valuable
%  comments and remarks. This work was partially supported by the National Natural Science Foundation of China (Nos. 11601286 and 11501146). }

%%%%%%%%%%%%%%%%%%%%%%%%%%%%%%%%%%%%%%%%%%%%%%%%%%%%%%%%%%%  reference   %%%%%%%%%%%%%%%%%%%%%%%%%%%%%%%%%%%%%%%%%%%%%%%

\end{document}